\documentclass[abstracton, final, paper=a4, fontsize=11pt,DIV=14,bibliography=totoc]{scrartcl}

\pdfoutput=1

\usepackage[utf8]{inputenc}
\usepackage[T1]{fontenc}
\usepackage{lmodern,dsfont}
\usepackage[english]{babel}
\usepackage{amsthm, amssymb, amsmath, amsfonts, mathrsfs, dsfont, esint,textcomp}
\usepackage[colorlinks=true, pdfstartview=FitV, linkcolor=blue, citecolor=blue, urlcolor=blue,pagebackref=false]{hyperref}
\usepackage{mathtools}
\usepackage[shortlabels]{enumitem}

\usepackage[backend=biber,giveninits,maxnames=4,style=alphabetic,isbn=false, date=year, doi=false, eprint= true, url= false]{biblatex}

\usepackage{a4wide}  %\usepackage{fullpage}
\usepackage{microtype}
\usepackage{csquotes}

 \usepackage{authblk}

\setkomafont{date}{\large}

\usepackage[notcite,notref,color]{showkeys}
\definecolor{labelkey}{gray}{0.7}
\definecolor{refkey}{gray}{0.7}

\definecolor{darkblue}{rgb}{0,0,0.7} 
\definecolor{darkgreen}{rgb}{0,0.5,0}

\newcommand{\mh}[1]{{\color{darkblue}{#1}}}

\newtheorem{thm}{Theorem}[section]
\newtheorem{prop}[thm]{Proposition}
\newtheorem{lem}[thm]{Lemma}

\theoremstyle{remark}
\newtheorem{rem}[thm]{Remark}
\theoremstyle{definition}

\definecolor{green}{rgb}{0,0.5,0}

\renewcommand{\leq}{\leqslant}
\renewcommand{\geq}{\geqslant}

\renewcommand{\subset}{\subseteq}

\newcommand{\N}{\mathbb{N}}

\newcommand{\1}{\mathds{1}}
\newcommand{\R}{\mathbb{R}}

\newcommand{\Z}{\mathbb{Z}}

\newcommand{\dd}{\, \mathrm{d}}
\newcommand{\wto}{\rightharpoonup}

\newcommand{\St}{\mathrm{St}}

\DeclareMathOperator{\dist}{dist}

\DeclareMathOperator{\curl}{curl}

\DeclareMathOperator*{\dv}{div}

\DeclareMathOperator*{\supp}{supp}

\numberwithin{equation}{section}
\bibliography{Hoefer.bib}

\begin{document}

\author[1]{Matthieu Hillairet\thanks{matthieu.hillairet@umontpellier.fr}}
\author[2]{Richard M. H\"ofer\thanks{richard.hoefer@ur.de}}
\affil[1]{Institut Montpelliérain Alexander Grothendieck, Université de Montpellier, France}
\affil[2]{Fakultät für Mathematik, Universität Regensburg, Germany}

\title{Hindered Settling of Well-Separated Particle Suspensions}
\date{\today}

\maketitle

\begin{abstract}
We consider $N$ identical inertialess rigid spherical particles in a Stokes flow in a domain $\Omega \subset \R^3$.  
We study the average sedimentation velocity of the particles when an identical force acts on each particle.
If the particles are homogeneously distributed in directions orthogonal to this force, then they hinder each other leading to a mean sedimentation velocity which is smaller than the sedimentation velocity of a single particle in an infinite fluid. 
Under suitable convergence assumptions of the particle density and a strong separation assumption, we identify the order of this hindering as well as effects of small scale inhomogeneities and boundary effects. For certain configurations we explicitly compute the leading order corrections.
\end{abstract}

%%%%%%%%%%%%%%%%%%%%%%%%%%%%%

%%%%%%%     Introduction  						%%%%%%%

%%%%%%%%%%%%%%%%%%%%%%%%%%%%%

\section{Introduction}
The sedimentation velocity of a single inertialess rigid sphere in an infinite fluid follows immediately from Stokes' law for the drag force.
This law entails that the sphere falls parallel to the direction of the force acting on the particle (say gravity) with amplitude:
\begin{align}  \label{eq:V.Stokes}
 	V^{\St} := \frac{|F|}{6 \pi \mu R},   
\end{align}
where $F$ is the force acting on the particle, $R$ its radius and $\mu$ the fluid viscosity.
When several particles fall in the flow, the possible interactions between the particles through the fluid make however the situation much more complicated
as soon as there are more than 3 particles, see  \cite[Section 6.1]{GuazzelliMorris12}.

When $F$ is gravity, computing the mean sedimentation velocity of a cloud of particles in a Stokes flow is then a classical problem that has been studied in many previous references \cite{Batchelor72Sedimentation,Burgers1938,Feuillebois84,GeigenmullerMazur88, Hasimoto59, Saffman73}, to mention a few. We refer to the review \cite{DavisAcrivos85}  and to the introduction of \cite{DuerinckxGloria22Sedimentation} for a historical perspective.
In these works it has been observed (mostly on a formal level) that the mean sedimentation velocity of a cloud of $N$ particles in the whole space {remains parallel to $F$ and that its magnitude $\bar{V}_N^{sed}$ behaves in fundamentally different ways dependent on the particle distribution.
\begin{enumerate}
	\item[(Dil)] There is a characterization of \textit{diluteness} of suspensions for which the settling particles behave as if they were alone in the fluid \cite{JabinOtto04}.

	\item[(MF)] If the particles are less dilute and not homogeneously distributed in directions orthogonal to gravity, a \textit{macroscopic fluid flow} is created which enhances sedimentation: for sufficiently regular particle distributions, where not too much clustering occurs, the mean sedimentation velocity is of order
	\begin{align} \label{eq:JabinOtto}
		\bar V_N^{sed} \sim \max\left\{V^\St,\frac{N F}{\mu L}\right\}
	\end{align}
	where $N$ is the number of particles and $L$ is the typical length scale of the particle cloud \cite{Hofer18MeanField, Mecherbet19}. The additional term $\frac{N F}{\mu L}$ is precisely the parameter that characterizes diluteness in the above sense for such regular distributions and  can be much larger than $V^{\St}$.
	
	\item[(HS)] If the particles are closer and homogeneously distributed in directions orthogonal to gravity, the incompressibility of the  fluid prevents the onset of a macroscopic fluid flow that enhances sedimentation. Instead, a small fluid backflow is created that \textit{hinders the particle sedimentation}.	
The order of this hindering is again sensitive to the particle 	distribution:

\begin{enumerate}
   	\item If the particles are  \textit{periodically} distributed, then
		\begin{align} \label{expansion.periodic}
		\bar V_N^{sed} = V^{\St} (1 - a_{per} \phi^{\frac 1 3} + o(\phi^{\frac 1 3}))
		\end{align}
		for some $a_{per} > 0$, where $\phi$ is the particle volume fraction inside the fluid \cite{Hasimoto59}.
		\item If the particles are  distributed according to \textit{hardcore Poisson process} with hardcore distance $2R$, then
		\begin{align} \label{batchelor}
		\bar V_N^{sed} = V^{\St} (1 - a_{uni} \phi + o(\phi))
		\end{align}
		for some $a_{uni} > 0$ \cite{Batchelor72Sedimentation}.
\end{enumerate}
\end{enumerate}

The expansion  \eqref{expansion.periodic} has been rigorously shown in \cite{Hasimoto59} on the torus.
In this contribution we show that {it persists} to hold asymptotically for large $N$ if the particles are placed in a container $\Omega \subset \R^3$ such that 
\begin{itemize}
	\item The particles respect a separation distance of order $N^{-1/3}.$
	\item The container $\Omega$ is bounded in directions orthogonal to the direction of the acting force and the particles are sufficiently close to a macroscopic density $n$ which is constant in directions orthogonal to the acting force. 
\end{itemize}
Although we are mainly interested in the (HS) situation, we complement the analysis in the case (MF) when the orthogonality assumption is not satisfied. 

{The influence of the container on the sedimentation has been studied on a formal level in several works, see \textit{e.g.} \cite{BeenakkerMazur85, GeigenmullerMazur88, BruneauFeuilleboisAnthoreHinch96}. In these works, the particles are distributed according to a hardcore Poisson process as in \cite{Batchelor72Sedimentation}. However, in contrast to \cite{Batchelor72Sedimentation} where the whole space is considered, a nonoverlapping condition with the boundary $\partial \Omega$ restricts the particle centers to lie in $\Omega_R = \{ x \in \Omega: \dist(x,\partial \Omega) > R\}$. Since the particles are spherical, this leads to a lower mean volume concentration of particles in $\Omega \setminus \Omega_R$ than in $\Omega_R$ (where this concentration is constant). This discrepancy leads to a macroscopic fluid flow $v_f$  just like in (MF). However, since the inhomogeneity only occurs in the small region $\Omega \setminus \Omega_R$, this macroscopic fluid flow, called \textit{intrinsic convection}, is much smaller than in (MF). The authors in \cite{BeenakkerMazur85, GeigenmullerMazur88, BruneauFeuilleboisAnthoreHinch96} obtain $v_f = O(\phi V^{\St})$. Moreover $v_f$ decreases the sedimentation speed of particles close to the boundary of the container while it increases the sedimentation speed of particles in the bulk.
In the present paper, we rigorously identify 
a related but quantitatively different effect. Namely, for particle configurations satisfying both items above, 
we analyze perturbations of the particle distributions on the $N^{-1/3}$-scale that occur in the bulk rather than at the boundary of the container. 
This leads to 
macroscopic fluid velocities  $v_f = O(N^{1/3}\phi^{1/3} V^{\St})$. 
The contribution of this macroscopic fluid velocity to the average sedimentation velocity is much lower though, namely of order  $\phi^{1/3} V^{\St}$.
}

\medskip

All these approaches  to the computation of sedimentation velocity (including the present contribution) are based on a similar construction of the many-particle Stokes solution. Acting a force on each particles entails a microscopic disturbance in the flow around the particle that decays very slowly to zero at infinity. Summing the microscopic disturbances of all the particles cloud on one particle then creates a macroscopic disturbance that modifies its sedimentation velocity. A key-difficulty is then to prove that, despite the slow decay of the microscopic distubances, the macroscopic disturbance remains bounded, motivating many of the previous references
on the topic. If the particles are sufficiently far one from the other then the macroscopic disturbance can be shown to be neglectible and we recover \cite{JabinOtto04}. 
While, if the particles are closer, it turns out that the macroscopic disturbance can be proved to be bounded only because of a backflow due to the fluid incompressibility. 
For instance,   in the case of particles on cubic lattices, Hasimoto mimicks the backflow  on the torus by imposing the constraint that the total fluid flow (after extending the fluid flow inside of the particles) vanishes. By Fourier analysis, he then explicitly computed the expansion \eqref{expansion.periodic} \cite{Hasimoto59}.

In this contribution, we show that the boundaries make the macroscopic disturbance converge: they induce naturally a normalization of the pressure that makes the backflow explicit and the microscopic disturbances due to each particle decay faster. This improves the simplicity of the analysis.

%
%
%
%\rhcomment{Discuss more the role of boundary effects.}
%
%\subsection{Related results}
%
%\rhcomment{Discuss in more detail Batchelor, Feuillebois?
%Mention Duerinckx-Gloria, Jabin-Otto etc., boundary effect as in Guazzelli-Morris. Should also discuss somewhere increased viscosity, in particular that this is a lower order effect in our setting.}

\subsection{Setting}

Let $\Omega \subset \R^3$ be of class ${\mathcal C}^2$ and contained in an infinite cylinder with an orientation $\xi$, \textit{i.e.},
\begin{align} \label{eq:Omega.bounded.orthogonal} \tag{H0}
\exists\, C_1 > 0,  \quad \text{ s.t.  } \quad  \Omega \subset \{ x \in \R^3 : \dist (x,{\text span}\{\xi\}) < C_1 \}  .
\end{align}
We point out that $\Omega$ might be bounded as well as unbounded.
For $N \in \N$ and $r > 0$, let
$R_N := N^{-1/3} r$ and
 $X^N_i \in \Omega$ such that $B^N_i := B_{R_N}(X^N_i) \Subset \Omega$
 and $\overline B^N_i \cap \overline B^N_j = \emptyset$ for all $1 \leq i \neq j \leq N$. We will write $R$, $X_i$ and $B_i$ instead of $R_N$, $X_i^N$ and $B_i^N$ in the following.    We assume throughout the paper that the distribution of particles is regular in the following sense. 
 Firstly, we have the following separation assumptions:
\begin{align} \label{eq:minimal.distance} \tag{H1}
	\exists \, c > 0 \quad  \min_{i \neq j} |X_i - X_j| \geq c N^{-1/3}, \qquad \min_{i=1,\ldots,N} \dist(X_i,\partial \Omega) \geq c N^{-1/3}.
\end{align} 
The key information here is that the constant $c$ does not depend on $N.$
Secondly, we assume that  the empirical measure 
 \begin{align} \label{eq:rho_N}
 	\rho_N = \frac 1 N  \sum_{i=1}^N \delta_{X_i}
\end{align} 
is close to a density $n \in \mathcal P(\Omega) \cap L^\infty(\Omega)$
where $\mathcal P(\Omega)$ denotes the space of probability measures on $\Omega$.  For this,  we impose the following control on the infinite Wasserstein distance:
 \begin{align} \label{eq:Wasserstein} \tag{H2}
 	\mathcal W_\infty (\rho_N, n) \leq C_0 N^{-1/3}.
 \end{align}
 Again, the key information here is that the constant $C_0$ is independent of the number of particles.  For simplicity, we assume that the cloud of particles is uniformly bounded, i.e.,
 \begin{align} \label{eq:K} 
 \exists \, K \Subset \overline \Omega,  ~ \forall \, i \in \{1,\ldots,N\}, ~ X_i \in K.
 \end{align}

Our goal in this paper is to derive information on the mean sedimentation velocity of the particles when they are submitted to a given force $F \in \R^3.$
Since we restrict to a linear Stokes problem,  we assume without restriction that $F$
is directed along the third vector $e_3$ of the canonical basis and we normalize its amplitude to $N^{-\frac 13}.$  In this way,  the Stokes velocity (cf. \eqref{eq:V.Stokes}) is independent of $N$, namely,
\begin{align} \label{V^St_r}
	V^{\St} = V_r^{\St}= \frac{1}{6 \pi r}.
\end{align}
We consider  then the problem 
\begin{align}
	\label{def:u_N}
	\left.
 \begin{aligned}
 	- \Delta u_N + \nabla p_N &= 0 && \text{in } \Omega \setminus \bigcup_{i=1}^N \overline {B_i}, \\
  \dv u_N &= 0 && \text{in } \Omega \setminus \bigcup_{i=1}^N \overline {B_i}, \\
  u_N &= 0 && \text{on } \partial \Omega,\\[6pt]
  u_N(x) &= V_i + \Omega_i \times (x - X_i) && \text {in } \overline{B_i} \quad \text{for all } 1 \leq i \leq N, \\
-  \int_{\partial B_i} \sigma[u_N, p_N] \nu &= N^{-\frac{1} 3} e_3 && \text{for all } 1 \leq i \leq N,\\
  - \int_{\partial B_i} (x-X_i)  \times \sigma[u_N, p_N] \nu   &= 0 && \text{for all } 1 \leq i \leq N, \\
 {\lim_{|x| \to \infty} u_N(x)} &{= 0}
 \end{aligned}
 \right\}
\end{align}
In this system, we recall that $\nu$ is the normal to $\partial B_i$ (directed inwards $B_i$). The symbol $\sigma$ stands for the fluid stress tensor given by Newton law:
\[
\sigma[u,p] = 2 D(u) - p \mathbb I_3  = (\nabla u + \nabla^{\top} u ) - p \mathbb I_3.
\]
Note that the first equation in \eqref{def:u_N} reads also:
\[
\dv (\sigma(u_N,p_N))= 0
\]
where the operator $\dv$ acts rowwise on the matrix ${\sigma}(u_N,p_N).$
The symbols $V_i$ and $\Omega_i$ stand respectively for the linear and angular velocities of particle $B_i.$ We emphasize that these velocities together
with $(u_N,p_N)$ are the unknowns in \eqref{def:u_N}.  The system is then algebraically well-posed,  the velocities $(V_i,\Omega_i)$ being the Lagrange multipliers of the two last equations in \eqref{def:u_N}.  In particular,  these velocities depend on $N$ but we skip the dependencies for legibility. 
The last condition in \eqref{def:u_N} is needed in the case when $\Omega$ is unbounded in order to rule out Poiseuille type flows. We will in the following not write this condition explicitly. We will only consider velocity fields in $\dot H^1(\Omega)$ though, and Poiseuille type flows are not contained in this space.

\medskip

We are interested in the average particle velocity 
\begin{align}
	\bar V_N := \frac 1 N \sum_{i=1}^N V_i.
\end{align}
for large $N$ under the assumption: 
\begin{align} \label{eq:gradient} \tag{Hom}
	\curl(n e_3) = \nabla n \times e_3 = 0.
\end{align}
This assumption is reminiscent of (HS).  We recall that, as mentioned in introduction,  if the limit density $n$ is not constant in the directions perpendicular to $e_3$ (namely, in case (MF)), the particles create a collective fluid velocity proportional to the number of particles $N$ and the magnitude of $\bar{V}_N$ scales differently in $N.$
The importance of this assumption can be observed as follows.  If the particles are small and their distribution dilute, the force acting on the particles is seen {reciprocally} by the fluid as a forcing term $f$ concentrated in the particles:
\[
f \sim  \sum_{i=1}^{N} 6\pi N^{-1/3} e_3 \delta_{X_i} \sim 6\pi N^{2/3} n e_3.
\]
For large $N$ we expect then that the leading term in the velocity-unknowns behaves like $N^{2/3} (u,p)$ with $(u,p)$ solution to 
\begin{equation} \label{eq_asymptoticu}
\left.
 \begin{aligned}
 	- \Delta u + \nabla p  &= 6\pi  n e_3  && \text{in } \Omega, \\
  \dv u &= 0 && \text{in } \Omega  \\
  u &= 0 && \text{on } \partial \Omega.
  \end{aligned}
  \right\}
\end{equation}
One may then  expect that the mean velocity $\bar{V}_N$ has magnitude  $N^{2/3}$ unless $u=0.$ In this latter case,  we must have that $ne_3$ is a gradient or equivalently that  \eqref{eq:gradient} holds true.
Even when \eqref{eq:gradient} holds true,  it will appear that the components of $\bar{V}_N$ have different magnitudes.  Below, we call sedimentation velocity the projection of $\bar{V}_N$ along $e_3$: 
\[
\bar{V}^{sed}_N = \bar{V}_N \cdot e_3.
\]

To end this subsection, we point out that \eqref{eq:gradient} together with \eqref{eq:Omega.bounded.orthogonal} entail that,  if the axis $\xi$ and the force $e_3$ are orthogonal, then $n$ is necessarily constant in the direction $\xi$ which contradicts that $n$ is a probability measure. This is not the situation that we are interested in here.

\subsection{Main results}

For $N$ fixed,  the system \eqref{def:u_N} is well posed \textit{via} the following construction. A classical framework is the space of extended
velocity-fields:
\begin{equation} \label{eq_defHN}
H_0[N] := \{ w \in H^1_0(\Omega) \text{ s.t.  ${\textrm div}\ w =0$  on $\Omega$ and } D(w) = 0 \text{ on $B_i^N$ for all $i$}\}
\end{equation}

We remind that,  since the $B_i$ are connected, for arbitrary $w \in H_0[N]$ there exists vectors $(W_1,\ldots,W_N)$ and vectors $(R_1,\ldots R_N)$ so that:
\[
w(x) = W_i + R_i \times (x-X_i)\,, \quad \forall \, x \in B_i .
\]
In particular,  an extended velocity-field $w \in H_0[N]$ encodes $u_N$ but also $(V_i,\Omega_i)_{i=1,\ldots,N}.$ Classically, we only need to compute these unknowns to solve our system since the pressure $p_N$ is then recovered as the Lagrange multiplier of the divergence-free constraint. 
Eventually,  we have the weak formulation of \eqref{def:u_N}:
\begin{center}
\begin{minipage}{.8\textwidth}
Find $u_N \in H_0[N]$ such that,
\[
\int_{\Omega} \nabla u_N : \nabla w = \sum_{i=1}^{N}
\frac{e_3 \cdot W_i}{N^{1/3}}\,, \quad \forall \, w \in H_0[N].
\]
\end{minipage}
\end{center}
Such a weak formulation is obtained by mutliplying formally the Stokes equation with $w$ and performing integration by parts to apply (pointwise and integral) boundary conditions on $u_N.$ From this weak formulation, we immediately deduce
\begin{align} \label{eq:energy.mean.velocity}
	\|\nabla u_N\|_{L^2(\Omega)}^2 = \sum_{i=1}^N \frac{e_3 \cdot V_i}{N^{1/3}} = N^{2/3} \bar{V}_N^{sed}.
\end{align}

{We see on this energy identity that there is a non-trivial relationship between $u_N$ and $\bar{V}_N.$ One could have expected that  the sedimentation velocity  $\bar{V}_N^{sed}$ is of the same order (with respect to $N$) as the fluid velocity $u_N$ itself. 
The energy identity, however, relates the sedimentation velocity  $\bar{V}_N^{sed}$ to the \emph{gradient} of the fluid velocity $u_N$ and reveals a factor $N^{2/3}$ between $\|\nabla u_N\|^2_{L^2(\Omega)}$ and $\bar{V}_N^{sed}.$} Our first main result is then the identification of the magnitude of $\bar{V}_N$ in both cases when \eqref{eq:gradient} holds true and does not hold true:

\begin{thm} \label{th:scaling.N}
Assume that  \eqref{eq:Omega.bounded.orthogonal}--\eqref{eq:Wasserstein}  are satisfied.

\begin{enumerate}[(i)]
\item \label{it:thm.unfavorable}  Assume that \eqref{eq:gradient} is \textit{not} satisfied.  Then, there exists $C$ depending only on  $\Omega$, on $n$  and on $C_0,c$, from \eqref{eq:Wasserstein} and \eqref{eq:minimal.distance} such that
	\begin{align}
	\limsup_{N \to \infty} N^{-\frac 2 3}|\bar V_N| \leq C  , \label{est:V.unfavorable}\\
	 \liminf_{N\to \infty}  N^{-\frac 23} \bar V_N^{sed} \geq \frac 1 C \label{est:V^sed.unfavorable}.
	\end{align}
	
\item \label{it:thm.favorable} If \eqref{eq:gradient} is satisfied then  there exists $C$ depending on  $\Omega$  and on $C_0,c$ such that
	\begin{align}
\limsup_{N \to \infty} N^{-\frac 13}	|\bar V_N| \leq C r^{-1/2} , \label{est:V.favorable}\\
	\limsup_{N \to \infty} |\bar V_N^{sed} -  V^{St}_r| \leq C.	\label{est:V^sed.favorable}
	\end{align}
\end{enumerate}
\end{thm}

 We remark that the factor $r$ is related to the volume fraction $\phi$ through $\phi \sim r^3.$ For instance,  if $n$ is the indicator of some connected open set $K \Subset \Omega$ with    $|K|=1$ (say a unit cube for instance), we can compute a local volume fraction $\phi = 4\pi r^3/3.$   We recall also that $V^{\St}_r \sim r^{-1}$.   In particular, since $r$ is independent of $N$ we have $N^{2/3}  \gg   V^{St}$ for $N \gg 1$, and therefore,  in case \eqref{eq:gradient} is not satisfied, \eqref{est:V^sed.unfavorable} is coherent with \eqref{eq:JabinOtto}. 
In case \eqref{eq:gradient} holds true,  \eqref{est:V^sed.favorable} is a prerequisite in order that an expansion \eqref{expansion.periodic} can be valid. 

If \eqref{eq:gradient} holds true,  the solution to \eqref{eq_asymptoticu} is a pure pressure.  With similar arguments as previously,  a more relevant approximation to $(u_N,p_N)$ for large $N$ is then  $N^{2/3}(\tilde{u},\tilde{p})$ where 
$(\tilde{u},\tilde{p})$ is the solution to:
\begin{equation} \label{eq_asymptoticutilde} \left.
 \begin{aligned}
 	- \Delta \tilde{u} + \nabla \tilde{p}  &=  6\pi (\rho_N - n) e_3  && \text{in } \Omega, \\
  \dv \tilde{u }&= 0 && \text{in } \Omega  \\
  \tilde{u} &= 0 && \text{on } \partial \Omega.
  \end{aligned} \right\}
\end{equation}
According to the rate of convergence \eqref{eq:Wasserstein}, one may then expect that the mean velocity $\bar{V}_N$ is of size $N^{1/3}.$ The even smaller size of the sedimentation velocity (in powers of $N$) comes from the remark that:
\[
\bar{V}_N^{sed}  = \bar{V}_N \cdot e_3 \sim \langle N^{2/3} \tilde{u} ,  \rho_N e_3 \rangle 
 = N^{2/3}\langle \tilde{u} , ( \rho_N  - n) e_3 \rangle 
\]
We used here again that, under assumption \eqref{eq:gradient},  the term
$n e_3$ is a pressure gradient.  The further gain of $N^{1/3}$ then yields from 
\eqref{eq:Wasserstein}  again.  This gain can be generalized to the component of $\bar{V}_N$ along any vector $e \in \mathbb S^2$ such that $\nabla n \times e=0.$

\medskip 

In order to derive and characterize an expansion of the form \eqref{expansion.periodic}, we introduce the two following additional structural assumptions.
The first assumption regards a refined convergence of $\rho_N$ to $n$. To this end, we first smooth out the density $\rho_N$ as follows
\begin{align} \label{eq:sigma_N}
	 \sigma_N := \frac 1 N \sum_{i=1}^N \frac 1 {|Q_i|} \1_{Q_i}\,, \qquad   \bar{\rho}_N = \dfrac{1}{N}\sum_{i=1}^N \dfrac{1}{|\partial B_i|} \mathcal H^{2}_{\partial B_i} .
\end{align}
Here  $\mathcal H^2_{\partial B_i}$ is the Hausdorff measure on $\partial B_i$ while  the $Q_i$ are disjoint cubes centered at $X_i$ of volume
\begin{align} \label{eq:Q_i}
 \frac{1}{C_1 N} \leq |Q_i| \leq \frac {C_1} {N}.
 \end{align} 
with $C_1$ independent of $N.$
 We emphasize that it is always possible  to find such cubes thanks to assumption \eqref{eq:minimal.distance} with
$C_1 = c^{-3}$.  However,  the $Q_i$ are not unique and we might change construction depending on the computations.
To characterize defects of $\rho_N$ to $n$, we impose that for a suitable choice of the cubes $Q_i$, the following strong convergence holds:
\begin{align} \label{ass:Strong.concergence} \tag{Str}
	N^{\frac 1 3}(\sigma_N - n) \to g \quad \text{in } H^{-1}(\Omega)   \qquad \text{for some } g \in H^{-1}(\Omega).
\end{align}
 We remark that by Proposition \ref{pro:Poincare} below $N^{\frac 1 3}(\sigma_N - n)$ is already bounded in $\dot H^{-1}(\Omega)$ under assumption \eqref{eq:Wasserstein}--\eqref{eq:minimal.distance}.

\medskip

The second assumption is an almost periodicity assumption on the particles.:
\begin{multline} \label{ass:Periodic} 
	 \exists \, d>0,  \; t_N \in \R^3,  \; E_N \subset \Omega, \; I_N \subset \{1, \dots, N \} \text{ s.t.}\notag \\
	 \left\{
	\begin{aligned} & |t_N|_\infty \leq N^{-1/3},  E_N \subset E_{N+1}, \frac {|I_N|}{N} \to 1,  \\
	& \{X_i : i \in I_N\} =  E_N \cap (t_N + d N^{-\frac 1 3 }\Z)^3, \\
	 & E_N = \bigcup_{i \in I_N} X_i + [-N^{-1/3}d,-N^{-1/3}d]^3 
	 \end{aligned} \right. \tag{Per}
\end{multline}
Here $E_N$ should be understood as the set on which the configuration is periodic, $I_N$ the set of particles which are periodically distributed and $t_N$ allows those particles to be uniformly translated with respect to a lattice centered at the origin.
We will give an example for a particle configuration that satisfies both \eqref{ass:Periodic} and \eqref{ass:Strong.concergence} with a nontrivial $g$ in Section \ref{sec:Str}.

\medskip

To give a characterization of the mean velocity, we introduce the following velocity fields.
We define $v_{N,1} \in \dot H^1(\R^3)$ as the solution of the following Stokes equations in the whole space $\R^3:$
\begin{align}
\left.
\begin{aligned}
	- \Delta v_{N,1} + \nabla p_{N,1} &= \sum_{i=1}^N \left(\frac 1 {|\partial B_R(X_i)|} \mathcal H^2|_{\partial B_R(X_i)} - \frac 1 {|Q_i|} \1_{Q_i} \right) e_3, \\
	\dv v_{N,1} &= 0.
\end{aligned} \right\}
\end{align}
Moreover, we consider the solution $v_{\infty,3} \in H^1_0(\Omega)$ to the Stokes equations in $\Omega$
\begin{align} \label{eq:v_infty,3} \left.
\begin{aligned}
 	- \Delta v_{\infty,3} +  \nabla p_{\infty,3} = g, 
 	\quad \dv v_{\infty, 3} &= 0 \qquad \text{in } \Omega, \\
 	  v_{\infty,3} &= 0 \qquad \text{on } \partial \Omega.
\end{aligned} \right\}
\end{align}
We keep the index $2$ for a further velocity-fields that we require for technical 
convenience below.  With these definitions, our expansion is the content of the following result:
\begin{thm} \label{th:expansion}
Assume that assumption \eqref{eq:Omega.bounded.orthogonal}--\eqref{eq:Wasserstein}  and \eqref{eq:gradient} are satisfied.

\begin{enumerate}[(i)]
\item \label{it:thm.strucutre} If in addition \eqref{ass:Strong.concergence} is satisfied, then, for all  $\delta > 0$, there exists $C > 0$, depending only on $C_0,C_1,c,\delta$ from \eqref{eq:Wasserstein}, \eqref{eq:minimal.distance} and \eqref{eq:minimal.distance} respectively such that
\begin{align} \label{structure}
	\limsup_{N \to \infty}\left| \bar V_N^{sed}  -   \left( N^{-2/3}  \|\nabla v_{N,1}\|_{L^2(\R^3)}^2 + \|\nabla v_{\infty,3}\|_{L^2(\Omega)}^2\right) \right| \leq C  r^{1 - \delta}.
\end{align}
{Moreover, 
\begin{align} \label{v_3,infty}
    \|\nabla v_{\infty,3}\|_{L^2(\Omega)}^2 = \lim_{N \to \infty} N^{1/3}  \langle v_{\infty,3}, \bar \rho_N e_3\rangle.
\end{align}
and there exists a sequence $w_N \in  H^1_0(\Omega)$ with $\|\nabla w_N\|_{L^2(\Omega)} \leq C N^{1/3} r^{3/2 - \delta}$ such that 
\begin{align} \label{weak.v_infty.3}
N^{-1/3} (u_N -w_N) \wto v_{\infty,3} \quad \text{weakly in }  H^1_0(\Omega).
\end{align}
}
\item \label{it:thm.periodic} If in addition \eqref{ass:Periodic} is satisfied {and $Q_i = X_i + [-N^{-1/3}d,N^{-1/3}d]^3$ for all $i \in I_N$,} then
\begin{align} \label{char.periodic}
	\lim_{N \to \infty}  N^{-2/3} \|\nabla v_{N,1}\|_{L^2(\R^3)}^2 = \|\nabla v_{per}\|_{L^2(\mathbb T_d^3)}^2 =  V^{\St}_r(1 - a_{per} \frac r d + o(r)),
\end{align}
for some constant $a_{per} > 0$ and where $v_{per}$ is the unique solution to 
\begin{equation} \label{eq_Stokesperiodique} \left.
\begin{aligned}
 	- \Delta v_{per} +  \nabla p_{per} = 
 	\left(\frac 1 {|\partial B_{r}|} \mathcal H^2|_{\partial B_{ r }(0)} - \1_{\mathbb T_d^3} \right) F &\qquad \text{in } \mathbb T_d^3, \\
 	\quad \dv v_{per} = 0 &\qquad \text{in } \mathbb T_d^3, \\
 	\int_{\mathbb T_d^3} v_{per} \dd x &= 0,
\end{aligned} \right\}
\end{equation} 
where $\mathbb T_d^3 = \R^3/(d \Z)^3$.
\end{enumerate}
\end{thm}

{A few remarks are in order. We first recall, in order to compare with the expansions for $\bar V^{sed}_N$  discussed at the beginning of the introduction, that  $r \sim \phi^{1/3}$.
The estimate \eqref{structure} characterizes the sedimentation velocity $\bar V_N^{sed}$ up to an $O(r^{1-\delta})$ error as the sum of two contributions. The first contribution, encoded in $v_{N,1}$, only depends on the particle configuration. It is completely independent of the container $\Omega$. Under the periodicity assumption \eqref{ass:Periodic}, we characterize this contribution  in \eqref{char.periodic} as the sum of the Stokes velocity  $V^{\St}_r$ and a correction of order $r V^{\St}_r$ that can be computed from the problem on the torus \eqref{eq_Stokesperiodique}. Recall from \eqref{V^St_r} that $r |V^{\St}_r| = O(1)$.

The second contribution to $\bar V_N^{sed}$ in \eqref{structure} is encoded in $v_{\infty,3}$. Note that $v_{\infty,3}$ is independent of $r$ and therefore the contribution $\|\nabla v_{\infty,3}\|_{L^2(\Omega)}$ is of order $1$ if $\curl g \neq 0$.
Moreover, the characterization \eqref{weak.v_infty.3} means that $v_{\infty,3}$ is the leading order normalized macroscopic fluid flow and by \eqref{v_3,infty} the contribution $\|\nabla v_{\infty,3}\|_{L^2(\Omega)}$ equals  the average of this leading order macroscopic fluid flow in the particles.
Note that even though the macroscopic fluid flow is of order $N^{1/3}$, \eqref{v_3,infty} implies that its average at the particles is of order $1$.
}

We also remark that the constant $a_{per}$ corresponds to the one from \eqref{expansion.periodic} analyzed in \cite{Hasimoto59}. We do not investigate further the computation of  $ \|\nabla v_{N,1}\|_{L^2(\R^3)}^2$ for particle configurations other than those satisfying \eqref{ass:Periodic}. One could expect though that the energy $\|\nabla v_{N,1}\|_{L^2(\R^3)}^2$ can be generally expressed in terms of the $2$-point correlation, similar as for the second order correction of the effective viscosity of a suspension obtained in \cite{Gerard-VaretHillairet19, DuerinckxGloria20}.

{
\subsection{Organization of the remainder of the paper and notations}

The remainder of the paper is devoted to the proofs of Theorem \ref{th:scaling.N} and \ref{th:expansion}.

Section \ref{sec:densities}, contains preliminary investigations on the probability densities involved in the analysis, namely the empirical measure of the particles smeared out to $\partial B_i$, the measure $\sigma_N$ and the limit density $n$. Section \ref{sec:Poincare} contains estimates between these densities which will be crucial for the subsequent analysis. In Section \ref{sec:Str}, we provide an example for assumption \eqref{ass:Strong.concergence} with a nontrivial function $g$.

In Section \ref{sec:favorable} we prove Theorem \ref{th:scaling.N} \ref{it:thm.favorable} as well as Theorem \ref{th:expansion} \ref{it:thm.strucutre}. The proof is based on the splitting of $u_N$ into $v_{N,1}, v_{N,2}, v_{N,3}$ and $w_N$ that account for a whole space solution, boundary corrections, the defect between the measures $\sigma_N$ and $n$ as well as higher order hydrodynamical interactions between the particles.

In Section \ref{sec:periodic} we show Theorem \ref{th:expansion} \ref{it:thm.periodic} by analyzing periodic particle configurations.

Finally, in Section \ref{sec:unfavorable}, we give the proof of Theorem  \ref{th:scaling.N} \ref{it:thm.unfavorable} that concerns the mean particle velocity in the ill-prepared case, when \eqref{eq:gradient} is not satisfied. We complement the proof by additional structural information, namely the strong convergence $v_N \to v_\ast$ in $H^1_0(\Omega)$ of the leading part $v_N$ of $u_N$  in terms of the particle volume fraction $r^3$ as well as a characterization of the leading order of the limiting behavior of the mean velocity $\bar V_N$ in terms of $v_\ast$.}

\medskip

In what follows, we use classical notations for function spaces.  We do not specify whether we handle vector or scalar functions. This shall be clear in the context.  If $U \subset \R^3$ is bounded,
we denote 
\[
\fint_{U} f(x) {\text d}x = \dfrac{1}{|U|} \int_{U} f(x){\text d}x \quad 
\forall \,  f \in L^p(U),
\]
and $L^p_0(U)$ the subset of $L^p(U)$ containing mean-free functions. 
Such definitions may be generalized to functions defined on hypersurface of $\mathbb R^3.$
Finally, for  arbitrary $U \subset \mathbb R^3, $ we denote 
\[
\dot{H}^1(U) = \{ u \in L^6(U)  \text{ s.t. } \nabla u \in L^2(U)\}. 
\] 
If $U$ is bounded we have $\dot{H}^1(U) = H^1(U)$ that we endow with the classical norm.
If $U$ is unbounded we endow $\dot{H}^1(U)$ with the norm 
\[
\|u\|_{\dot{H}^1(U)} = \|\nabla u\|_{L^2(U)}
\]
for which it is also a Hilbert space.

Below we use also constantly the symbol $\lesssim$ for an inequality involving a harmless 
(multiplying) constant.

%%%%%%%%%%%%%%%%%%%%%%%%%%%%%

%%%%%%%     Poincare 						%%%%%%%

%%%%%%%%%%%%%%%%%%%%%%%%%%%%%

\section{Properties of (smoothened-)empirical measures} \label{sec:densities}
In our problem, particle distributions are encoded:
\begin{itemize}
\item {\em via} the associated empirical measures $\rho_N$ at the discrete level,
\item {\em via} the density $n$ in the continuous model. 
\end{itemize}
For technical convenience, we need in the sequel smoothened versions of $\rho_N.$ Namely, we will use:
\begin{align} \label{def.sigma_N}
    \sigma_{N} = \dfrac{1}{N} \sum_{i=1}^N \dfrac{1}{|Q_i|}\mathds{1}_{Q_i} \qquad \bar{\rho}_N = \dfrac{1}{N}\sum_{i=1}^N \dfrac{1}{|\partial B_i|} \mathcal H^{2}_{\partial B_i}  
\end{align}
where we recall that $Q_i$ are cubes centered in the $X_i$ of volume scaling like $1/N$ (see assumption \eqref{eq:Q_i}) while $\mathcal H^2_{\partial B_i}$ is the Hausdorff measure on $\partial B_i.$ 
In this section we prove at first some preliminary Poincaré type estimates that are crucial for the later analysis.
These inequalities enable to control distances between smoothened empirical measures and between empirical measures and their continuous conterparts.  We provide then examples of particle distributions for which
assumption \eqref{ass:Strong.concergence} holds true with an explicit $g.$

\medskip

\subsection{Poincaré type inequalities} 
\label{sec:Poincare}
 
The first purpose of this section is the following estimates regarding particle distributions:

\begin{prop} \label{pro:Poincare}
Let $p \in (1,\infty)$ and assume that $B_i \subset Q_i$ for all $i.$
%
%\rhcomment{What about $p=1,\infty$?}
\begin{enumerate}[(i)]
\item \label{it:rhoN.sigma_N}
If $p \neq 3$, there exists a constant $C$ that depends only on $p$ and the constant $C_1$ from \eqref{eq:Q_i} such that
\begin{align} \label{est:rhoN.sigma_N}
\|\bar \rho_N - \sigma_N\|_{(W^{1,p}(\Omega))^*} \leq C r^{-(3/p-1)_+} N^{- \frac 13}, 
\end{align}
and, if $p > 3/2$,
\begin{align}
\|\bar \rho_N -  \sigma_N\|_{(W^{2,p}(\Omega))^*} \leq C  N^{-2/3}, \label{est:rhoN.sigma_N.2}
\end{align}
where $(\cdot)_+$ stands for the positive part of real numbers.

\item  \label{it:n.sigma_N}
If $p \neq 3$  there exists a constant $C$ that depends only on $p$ and the constants $C_0$, $c$ and $C_1$ from \eqref{eq:Wasserstein}, \eqref{eq:minimal.distance} and \eqref{eq:Q_i} such that
\begin{align} \label{est:n.sigma_N}
\|\sigma_N - n\|_{(W^{1,p}(\Omega))^*} \leq C N^{-1/3},
\end{align}

\item \label{it:n.bar.rho_N}
If $p \neq 3$ there exists a constant $C$ that depends only on $p$,  $n$ and the constants  $C_0$ and $c$ from \eqref{eq:Wasserstein} and 
\eqref{eq:minimal.distance} such that
\begin{align} \label{est:n.bar.rho_N}
\|\bar \rho_N - n\|_{(W^{1,p}(\Omega))^*} \leq C r^{-{(3/p-1)_+}} N^{-1/3},
\end{align}
\end{enumerate}

\end{prop}

We note that item $(i)$ entails in particular that 
 for all $p \in (1,\infty)$ 
\begin{align} \label{eq:convergence.rho_N.sigma_N}
N^{1/3} (\bar \rho_N -  \sigma_N) \wto 0 \qquad \text{in }  {(W^{1,p}(\Omega))^*}. 
\end{align}
It might be surprising that the scale in $N$ changes between 
\eqref{est:rhoN.sigma_N} and \eqref{est:rhoN.sigma_N.2} making \eqref{est:rhoN.sigma_N} seem far from optimal.  It must be noted, though, that by symmetry, all affine functions tested on $\sigma_N - \bar{\rho}_N$ vanish.  We will then obtain our result by comparing expansions of test-functions around each center $X_i.$ 
The discrepancy between both estimates is due to the fact that only zero-order expansions are available in $W^{1,p}$ while first-order expansions are available in $W^{2,p}.$
Finally,  inequality \eqref{est:rhoN.sigma_N.2} in case $p > 3/2$ could be complemented with a similar inequality in case $p < 3/2.$ This will be however useless to our purpose.

\medskip

For the proof, we furthermore introduce 
\begin{align}
	 \tilde \rho_N := \frac 1 N \sum_{i=1}^N \frac 1 {|B_i|} \1_{ B_i},
 \end{align}
and we first show the following estimates involving $\tilde \rho_N$:
\begin{lem}  \label{lem:Poincaré}
Let $p \in [1,\infty]\setminus \{3\}$ and assume that $B_i \subset Q_i$ for all $i.$
\begin{align}
\|\tilde \rho_N -  \sigma_N\|_{(W^{1,p}(\Omega))^*} \leq C r^{-(3/p-1)_+} N^{-1/3},
\end{align}
where $C$ depends only on $p$ and the constant $C_1$ from
\eqref{eq:Q_i}.
\end{lem}

\begin{proof}
We start with $p < 3$. Then,  we may use the continuous embedding $W^{1,p}(Q_i) \subset L^{p_*}(Q_i)$, where $1/p^*= 1/p - 1/3,$ which implies here that for any  $\varphi \in W^{1,p}(Q_i) \cap L^p_0(Q_i)$ we have:
\[
\|\varphi\|_{L^{p_*}(Q_i)} \leq C(C_1,q) \|\nabla \varphi\|_{L^q(Q_i)}
\]  
with a constant $C(C_1,p)$ independent of $\varphi$ by a straghtforward homogeneity argument. Consequently, we have for all $v \in W^{1,p}(\Omega),$ via a sequence of discrete and continuous H\"older inequalities:
\begin{align*}
|\langle \tilde \rho_N -  \sigma_N, v \rangle| & 
=  \dfrac{3}{N 4\pi R^3} \sum_{i=1}^N \int_{B_i}  \left( v(x) -  \fint_{Q_i} v(z){\text d} z \right) \text{d}x \\
& \lesssim \dfrac{1}{r^3} \sum_{i=1}^N |B_i|^{1-\frac 1{p_*}}\|\nabla v \|_{L^p(Q_i)} \lesssim \dfrac{1}{r^3} r^{3(1-\frac 1{p_*})} N^{1-\frac 1p - (1- \frac 1{p_*})} \|\nabla v\|_{L^p(\Omega)}\\
& \lesssim \dfrac{1}{r^{\frac{3}{p_*}}} N^{\frac 1{p_*}- \frac 1p} \|\nabla v\|_{L^{p}(\Omega)}.
\end{align*}
We conclude by recalling that $1/p^*= 1/p - 1/3.$

\medskip

In the case $p>3$ we have the embedding $W^{1,p}(Q_i) \subset C^{0,\theta}(Q_i)$ with $\theta = 1/3-1/p.$ This implies here that, for arbitrary 
$\varphi \in W^{1,p}(Q_i) \cap L^p_0(Q_i)$ we have,  by standard homogeneity arguments:
\[
\|\varphi\|_{L^{\infty}(Q_i)}\leq \dfrac{C(p,C_1)}{N^{\frac 13- \frac 1p}}  \|\nabla \varphi\|_{L^p(Q_i)} 
\]
with a constant $C(p,C_1)$ depending only on $p$ and $C_1$. By standard arguments, we have then that:
\[
|\langle \tilde \rho_N -  \sigma_N, v \rangle| \leq \dfrac{1}{N} N^{-(\frac 13 - \frac 1p)}\sum_{i=1}^N \|\nabla v\|_{L^{p}(Q_i)} \leq \dfrac{1}{N} N^{-(\frac 13 - \frac 1p)} N^{1-\frac 1p} \|v\|_{W^{1,p}(\Omega)}.
\]
This finishes the proof of \eqref{est:rhoN.sigma_N}.
\end{proof}

We are then in position to prove our main result.

\begin{proof}[Proof of Proposition \ref{pro:Poincare}]
The convergence \eqref{est:n.bar.rho_N} follows from \eqref{est:rhoN.sigma_N} and \eqref{est:n.sigma_N}.  It remains to show \eqref{est:rhoN.sigma_N}, \eqref{est:rhoN.sigma_N.2} and \eqref{est:n.sigma_N}.

\emph{Step 1: Proof of \eqref{est:rhoN.sigma_N}:}
Our result follows also immediately from Lemma \ref{lem:Poincaré} and the standard Poincaré-like inequality
	\begin{align}
		\|v - \fint_{\partial B_i} v \|_{L^p(B_i)} \leq C_p R \|\nabla v \|_{L^p(B_i)},
	\end{align}
	where $C_p$ depends only on $p \in (1,\infty)$.  Indeed, we split $\bar{\rho}_N - \sigma_N = \bar{\rho}_N - \tilde{\rho}_N + \tilde{\rho}_N - \sigma_N.$
	The second part is estimated \textit{via} the previous lemma while for the first part, we have:
	\begin{align}
		|\langle \tilde \rho_N - \bar \rho_N, v \rangle| &\leq \frac 1 N \sum_{i=1}^N \fint_{B_i} \left|v - \fint_{\partial B_i} v \right| \\
		& \leq \left(\frac 1 N \sum_{i=1}^N \fint_{B_i} \left|v - \fint_{\partial B_i} v \right|^p \right)^{\frac 1 p} \leq C_p N^{- \frac 1 p } R^{1- \frac 3 p}  \|\nabla v\|_{L^p(\Omega)} \\
		& = C_p N^{-\frac 13} r^{1- \frac 3 p} \|\nabla v\|_{L^p(\Omega)}.
	\end{align}

\emph{Step 2: Proof of \eqref{est:rhoN.sigma_N.2}:}
	The argument is  analogous as the proof of Lemma \ref{lem:Poincaré} in the case $p > 3$.
	Indeed, we observe that due to the assumption that $Q_i$ is centered in $X_i$, we have
	\begin{align} \label{eq_expansionordre2}
		\langle \bar \rho_N - \sigma_N, v \rangle & 
=  \dfrac{1}{N 4\pi R^2} \sum_{i=1}^N \int_{\partial B_i}  \left( v(x) -  \fint_{Q_i} v(z){\text d} z - \fint_{Q_i} \nabla v(z){\text d} z \cdot (x- X_i)  \right) \text{d}x. 
	\end{align}
		Moreover, for all $p > 3/2$ and all 
			$\varphi \in W^{2,p}(Q_i)$ satisfying
			\[
			\int_{Q_i} \varphi = 0 \qquad \int_{Q_i} \nabla \varphi = 0
			\] 
			we have,  by a standard homogeneity argument
\[
\|\varphi\|_{C^0(\bar{Q}_i)}\leq \dfrac{C(p,C_1)}{N^{\frac 2 3- \frac 1p}}  \|\nabla^2 \varphi\|_{L^p(Q_i)} 
\]
where $C(p,C_1)$ depends only on $p$ and $C_1$ from \eqref{eq:Q_i}.  This entails that:
\[
|\langle \bar \rho_N - \sigma_N , v \rangle | \leq  \dfrac{C(p,C_1)}{N^{1-1/p} N^{2/3}} \sum_{i=1}^N \|\nabla^2 v\|_{L^p(Q_i)}. 
\]
The assertion then follows again from application of the discrete Hölder inequality. 

\emph{Step 3: Proof of \eqref{est:n.sigma_N}:}
We observe that by the triangle inequality, assumption \eqref{eq:Wasserstein} and the definition of $ \sigma_N$, we have
 \begin{align}
\mathcal  W_{\infty}( \sigma_N, n) \leq \mathcal W_{\infty}( \bar{\sigma}_N, \rho_N) + \mathcal W_{\infty}(\rho_N, n)  \leq C N^{-1/3}.
 \end{align}
 where $\mathcal W_{\infty}$ is the Wasserstein distance built on the $sup$-norm.
 By definition, the first term on the right-hand side is bounded by $1/N^{1/3}.$
Now the desired estimate follows from the  result
%\rhcomment{Include reference to Santambriogio} that for all $1 < p < \infty$ and all $\mu,\nu \in \mathcal P(\Omega) \cap L^\infty(\Omega)$ \rhcomment{Does this hold for the endpoints?}
\begin{align}
	\|\mu - \nu\|_{(W^{1,p})^*} \leq C (\|\mu\|_\infty +  \|\nu\|_\infty)^{1/p} \mathcal W_{\infty}(\mu, \nu).
\end{align}
see \cite[Exercise 38]{Santambrogio15} and \cite[Proposition 5.1]{HoferSchubert21}. This concludes the proof.
\end{proof}

\subsection{Explicit construction of distributions satisfying \texorpdfstring{\eqref{ass:Strong.concergence}}{Str}} \label{sec:Str}

We focus now on the construction of an example of particle distributions so that \eqref{ass:Strong.concergence} holds true: 
\[
N^{\frac 13}(\sigma_N - n) \text{ converges in $H^{-1}(\Omega).$}
\]
To this end,  we consider the case $\Omega = {(-1,1) \times (0,1)} \times \mathbb R.$

Fix $M \in \mathbb N^*$ and $N=2 M^{3}.$ 
Firstly, we distribute $N/2$ particles covering $(0,1)^2.$ For this,
we  construct the cubes  $\tilde{Q}_k$ ($k \in \{0,\ldots,M-1\}^3$) with centers in $\tilde{X}_k  = 1/M(k_1+1/2,k_2+1/2,k_3+1/2),$ radius $1/M$ and thus volume $2/N.$ We choose then $\lambda \in (0,1/2)$ and set
$X_k = \tilde{X}_k - \lambda/M e_1.$
\[
Q_k = 
\left\{
\begin{aligned}
& \text{cube with center $X_k$ and radius $1/(2M)$ if $k_1 >1$}\\
& \text{cube with center $X_k$ and radius $1/(2M)-\lambda/M$ if $k_1=0$} 
\end{aligned}
\right.
\]
The remaining particles and cubes are obtained by  transforming the $X_k$ with the symmetry $\sigma_1$ with respect  to the plane $\{x_1=0\}.$
{One easily checks that \eqref{eq:Wasserstein} is satisfied for $n = \frac{1}{2}\mathds{1}_{(-1,1) \times (0,1)^2} $ by considering the transport map $T(x) = X_k$ for $x \in \tilde Q_k$.}
%\begin{align*}
%\rho_N & = \dfrac{1}{N} \sum_{i} \delta_{X_i} = \dfrac 1 2 \left(  \dfrac{2}{N} \sum_{k \in \{0,;..,M-1\}^3}  \delta_{X_i} +  \dfrac{2}{N} \sum_{k \in \{0,;..,M-1\}^3}  \delta_{\sigma_1(X_i)}  \right)  \\
%& \to \dfrac{1}{2}\mathds{1}_{(-1,1) \times (0,1)^2} \text{ in $\mathbb P(\Omega)$}.
%\end{align*}
Note that $n= \frac 12 \mathds{1}_{(-1,1) \times (0,1]^2}$ satisfies \eqref{eq:gradient}.
Explicit computations then show that, denoting $\hat{k} = (0,k_2,k_3)$ for arbitrary $(k_2,k_3) \in \{0,\dots, M-1\}^2$:
\[
N^{1/3} ( \sigma_N -  n )|_{(0,1)^2 \times \mathbb R} = 2^{-\frac 23} M \left( \left[ \dfrac{2}{N}\sum_{k_2,k_3=0}^{M-1} \dfrac{1}{|Q_{\hat k}|}\mathds{1}_{Q_{\hat k}} - \mathds{1}_{\{x_1 \in (0,(1 - \lambda)/M)\}}\right] -  \mathds{1}_{\{x_1 \in (1-\lambda/M,1)\}}\right)  . 
\]
Classical computations then entail that:
\[
\begin{aligned}
 M  \left[ \sum_{k_2,k_3=0}^{M-1}\dfrac{1}{M^3 |Q_{\hat k}|} \mathds{1}_{Q_{\hat k}} - \mathds{1}_{\{x_1 \in (0,(1 - \lambda)/M)\}}\right]  & \to \lambda \delta_{\{x_1=0\}} \text{ in $(W^{1,2}((0,1)^2 \times \mathbb R)^*$}\,,\\
M  \mathds{1}_{\{x_1 \in (1-\lambda/M,1 \}} & \to - \lambda \delta_{\{x_1= 1\}} \text{ in $(W^{1,2}((0,1)^2 \times \mathbb R)^*$}\,.
\end{aligned}
\]
{Using symmetry at $x_1 = 0$ and that $\delta_{\{x_1= 1\}} = 0$ in $H^{-1}((-1,1) 
\times (0,1) \times \mathbb R)$,} we deduce that \eqref{ass:Strong.concergence} holds true with $g= 2^{1/3} \lambda \delta_{x_1=0}$ in $H^{-1}((0,1)^2 \times \mathbb R).$ We see on this example that the term $g$ encodes  a finer description of the particle distribution.  Indeed,  we created artificially a distribution in which particles around $x_1=0$ are closer and thus have larger interactions.  Particles near $x_1=0$ will therefore be slowed down in comparison to the particles near $x_1=1.$ Such a difference will induce a variation of the velocity distribution in the cloud that is  captured by the term $v_{\infty,3}$ solution to \eqref{eq:v_infty,3}.

%%%%%%%%%%%%%%%%%%%%%%%%%%%%%

%%%%%%%    Scale In N %%%%%%%

%%%%%%%%%%%%%%%%%%%%%%%%%%%%%

\section{Computation  of \texorpdfstring{$\bar{V}^{sed}_N$}{V} when \texorpdfstring{\eqref{eq:gradient}}{grad} holds true} \label{sec:favorable}

Throughout this section,  we assume that  \eqref{eq:Omega.bounded.orthogonal}--\eqref{eq:Wasserstein} and \eqref{eq:gradient} are satisfied. 
%
%\rhcomment{Repeat assumptions in all statements anyway?}
%
Let $(u_N,p_N)$ be the solution to \eqref{def:u_N}. 
 We remind the definition of $\bar \rho_N$ from \eqref{def.sigma_N} introduce $v_N$ as the solution to 
\begin{align} \label{def:v_N}
\left. \begin{aligned}
 	- \Delta v_N + \nabla q_N &=  N^{\frac 2 3} \bar \rho_N e_3 \quad \text{in } \Omega, \\
  \dv v_N &= 0 \quad \text{in } \Omega, \\
      v_N &= 0 \quad \text{on } \partial \Omega.
 \end{aligned} \right\}
 \end{align}
and the remainder
\begin{align} \label{def:w_N}
 w_N := u_N - v_N.
\end{align}
We will estimate the contribution of $w_N$ through  the variational characterization of Stokes solution that entails $
	\|\nabla w_N\|_{L^2(\Omega)} \leq C \|D(v_N)\|_{L^2(\cup_i B_i)}.$
We therefore first turn to the analysis of $v_N$ itself.

\medskip

We furthermore remind the definition of $\sigma_N$ from \eqref{def.sigma_N} and split $v_N$ further into $v_N = v_{N,1} + v_{N,2} + v_{N,3}$ (resp. $q_N = q_{N,1} + q_{N,2}+ q_{N,3}$) where 
\begin{equation} \left.
\begin{aligned} \label{def:v_N,1}
 	- \Delta v_{N,1} + \nabla q_{N,1} &=  N^{\frac 2 3}(\bar \rho_N - {\sigma}_N) e_3 && \text{in } \R^3, \\
  \dv v_{N,1} &= 0 && \text{in } \R^3,
 \end{aligned}\right\}
 \end{equation}
 \begin{equation} \left.
 \begin{aligned} \label{def:v_N,2}
 	- \Delta v_{N,2} + \nabla q_{N,2} & = 0 &&  \text{in } \Omega \\
  \dv v_{N,2} &= 0 && \text{in } \Omega, \\
   v_{N,2} &= -v_{N,1} && \text{on } \partial \Omega,
 \end{aligned}\right\}
 \end{equation}
   and 
   \begin{equation} \left.
\begin{aligned} \label{def:v_N,3}
 	- \Delta v_{N,3} + \nabla q_{N,3} &=  N^{\frac 2 3}({\sigma}_N - n) e_3&&  \text{in } \Omega, \\
  \dv v_{N,3} &= 0&& \text{in } \Omega, \\
  v_{N,3} &= 0&& \text{on } \partial \Omega.
 \end{aligned} \right\}
 \end{equation}
The identity $v_N = v_{N,1} + v_{N,2} + v_{N,3}$ holds because, due to assumption \eqref{eq:gradient}, the term involving $n$ on the right-hand side of \eqref{def:v_N,3} can be absorbed into the pressure: there exists a function $p_n \in L^2_{loc}(\Omega)$ such that $\nabla p_n = n e_3$.
 
\medskip
  
We will show the following properties of these functions.
\begin{prop} \label{pro:contributions.v_N}
	There exists a constant $C> 0$ depending only on $\Omega$, and on $C_0$ and $c$ from \eqref{eq:Wasserstein}--\eqref{eq:minimal.distance} as well as on  $C_1$ from \eqref{eq:Q_i} such that the following holds.
	\begin{enumerate}[(i)]
		\item \label{it:v_N,1} $N^{-1/3}v_{N,1} \wto 0$ weakly in $\dot H^1(\R^3)$ and $N^{-1/3} v_{N,1} \to 0$ strongly in $W^{1,p}(\R^3 \setminus \overline{\Omega})$ for all $p \in (1,\infty)$.
		Moreover,
		\begin{align}  \label{est:v_N,1}
			\left|V^{\St}_r - N^{-2/3} \|\nabla v_{N,1}\|^2_{L^2(\R^3)} \right| \leq C.
		\end{align}
 		\item 	\label{it:v_N,2} For $p=2$, 
	\begin{align} \label{est:v_N,2}
		 N^{-1/3} \|\nabla v_{N,2}\|_{L^{2}(\Omega)} \to 0. 
	\end{align}

	\item For $p=2,$
	\begin{align}\label{est:v_N,3-Hilb}
N^{-1/3} \|\nabla v_{N,3}\|_{L^2(\Omega)} \leq C.
\end{align}
	For all $p \in (1,\infty)$ and all bounded sets $\Omega' \subset \Omega$  there holds:  \label{it:v_N,3}
	\begin{align}\label{est:v_N,3}
	N^{-1/3} \|\nabla v_{N,3}\|_{L^p(\Omega')} \leq C'.
	\end{align}
	with $C'$ depending furthermore on $\Omega'.$
	If in addition \eqref{ass:Strong.concergence} is satisfied, then $N^{-1/3} v_{N,3} \to v_{\infty,3}$,  strongly in $H^1(\Omega)$, where $v_{\infty,3}$ is the solution to \eqref{eq:v_infty,3}.	

	\item  \label{it:v_N.in.particles} For all $\delta > 0$
	\begin{align}\label{est:v_N.in.particles}
	\limsup_{N \to \infty}	N^{-1/3}  \|D(v_{N})\|_{L^2(\cup_i B_i)} \leq C r^{3/2 - \delta}
	\end{align}				
	where the constant $C$ depends in addition on $\delta$.	
	\end{enumerate}
\end{prop}
The proof of this proposition is postponed to Subsection  \ref{sub_3.1}.

\subsection{Proof of \texorpdfstring{Theorem \ref{th:scaling.N}}{th} when \texorpdfstring{\eqref{eq:gradient}}{grad} holds true}

To treat the error $w_N$, we note that  $w_N$ can be associated to a pressure $\bar q_N$ to yield a solution to
\begin{align} \label{def:psi} \left.
\begin{aligned}
 	- \Delta \psi + \nabla q &= 0 && \text{in } \Omega \setminus \bigcup_{i=1}^N \overline{B_i}, \\
  \dv \psi &= 0 && \text{in } \Omega \setminus \bigcup_{i=1}^N \overline{B_i}, \\
    \psi &= 0 && \text{on } \partial \Omega,\\
  D(\psi) &= D(\varphi)  && \text {in }  B_i  \quad \text{for all } 1 \leq i \leq N, \\
  \int_{\partial B_i} \sigma[\psi, q] n &= 0  = \int_{\partial B_i} \sigma[\psi, q] n \times (x-X_i)&&  \text{for all } 1 \leq i \leq N.
 \end{aligned} \right\}
 \end{align}
with $\varphi = -v_N$. The estimate for $w_N$ then follows from the following standard estimate (see e.g. \cite[Equation (27)]{Gerard-VaretHoefer21})
\begin{prop}  \label{prop:variational.psi}
	Let $\varphi \in H^1(\cup_i B_i)$ and let $\psi$ be the solution to \eqref{def:psi}.
	Then
	\begin{align}
		\|\psi\|_{H^1(\Omega)} \leq C \|D (\varphi)\|_{L^2(\cup_i B_i)}
	\end{align}
	for a universal constant $C$.
\end{prop} 

We show how {Proposition \ref{pro:Poincare}}  and {Proposition \ref{pro:contributions.v_N}}  imply {Theorem \ref{th:scaling.N}}~\ref{it:thm.favorable} and {Theorem \ref{th:expansion}}~\ref{it:thm.strucutre}.

\begin{proof}[Proof of Theorem \ref{th:scaling.N} \ref{it:thm.favorable}]

	We fix the choice of the cubes $Q_i$ by $|Q_i| = c^3 N^{-1}$ where $c$ is the constant from \eqref{eq:minimal.distance}. In this way, dependencies on $C_1$ from  \eqref{eq:Q_i} become dependencies on $c$.

Using that $V_i = \fint_{\partial B_i} u_N$, we first note that, for arbitrary direction $e \in \mathbb S^{2},$ there holds:
\[
\bar{V}_N \cdot e = \langle \bar{\rho}_Ne ,  u_N \rangle
 \]
 Writing that $u_N = v_{N,1} + v_{N,2} + v_{N,3} + w_N$ and that $\bar{\rho}_N = \bar{\rho}_N - n + n ,$ we combine \eqref{est:v_N,1}-\eqref{est:v_N,2}-\eqref{est:v_N,3}  together with \eqref{est:n.bar.rho_N}  in case $p=2$ to yield \eqref{est:V.favorable}.

	Using assumption \eqref{eq:gradient} and the fact that $u_N$ is divergence free and that $V_i = \fint_{\partial B_i} u_N$, we rewrite
	\begin{align} \label{eq:expansion.0}
		\bar V_N^{sed} =  \langle (\bar \rho_N - n)e_3, u_N \rangle 
	\end{align}
	We recall the decomposition $u_N = v_N + w_N = v_{N,1} + v_{N,2} + v_{N,3} + w_{N}$. 
By \eqref{est:v_N.in.particles}, Proposition \ref{prop:variational.psi} and \eqref{est:n.bar.rho_N} with $p=2$
	\begin{align} \label{est:contribution.w}
	\limsup_{N\to \infty}  \left |\langle (\bar \rho_N - n)e_3,w_{N} \rangle \right|  \leq C_\delta r^{1-\delta}.
	\end{align}
	Moreover, using the Stokes equations that $v_N$ solves,
	\begin{align*} 
			 \langle (\bar \rho_N - n)e_3, v_N \rangle &  = 
 			 \langle -\Delta v_N + \nabla q_N , v_{N} \rangle =
			  N^{-2/3} \|\nabla v_N\|_{L^2(\Omega)}^2 
	\end{align*}
From \eqref{est:v_N,2} we infer that we have a remainder $rem_N$ going to $0$ as $N \to \infty$ such that:
\[
  \langle (\bar \rho_N - n)e_3, v_N \rangle =  N^{-2/3} \|\nabla(v_{N,1} + v_{N,3})\|_{L^2(\Omega)}^2 + rem_N .
\]
and thus:
\begin{align}  \label{v_N,1+v_N,3}
\limsup_{N \to \infty}	|\bar V_N^{sed} - V^{\St}_r| &\leq 
\limsup_{N \to \infty}	| N^{-2/3} \|\nabla (v_{N_1} + v_{N,3}) \|_{L^2(\Omega)}^2  - V^{\St}_r|  + C_{\delta} r^{1-\delta} 
\end{align} 
At this point, we realize that, with \eqref{est:rhoN.sigma_N} and \eqref{est:v_N,3} with $p>3$ 
\[
N^{-\frac 23}\int_{\Omega} \nabla v_{N,1} : \nabla v_{N,3} = \langle   (\bar{\rho}_N - \sigma_N) e_3, v_{N,3} \rangle \leq C .
\]
Therefore, expanding the square in \eqref{v_N,1+v_N,3}, and using also \eqref{est:v_N,3-Hilb}  yields
\[
\limsup_{N \to \infty}	|\bar V_N^{sed} - V^{\St}_r| \leq \limsup_{N \to \infty}	| N^{-2/3} \|\nabla v_{N_1}  \|_{L^2(\Omega)}^2  - V^{\St}_r|  +  C
\]
and we obtain the expected result thanks to \eqref{est:v_N,1}.
 \end{proof}

\medskip

\begin{proof}[Proof of Theorem \ref{th:expansion} \ref{it:thm.strucutre} ]
We now turn to the 
proof of Theorem \ref{th:expansion} \ref{it:thm.strucutre}. This time, we choose the cubes $Q_i$ such that assumption \eqref{ass:Strong.concergence} is satisfied. We revisit the latter computations,  using that Proposition \ref{pro:contributions.v_N} provides the weak convergence $N^{-1/3} v_{N,1} \wto 0$ in $\dot H^1(\Omega)$ and that, thanks to \eqref{ass:Strong.concergence},   we have the strong convergence $N^{-1/3} v_{N,3} \to v_{\infty,3}.$ We infer:
\begin{align} \label{eq:energy.sum.2}
\begin{aligned}
			\limsup_{N \to \infty} \langle (\bar \rho_N - n)e_3, v_N \rangle &= \limsup_{N \to \infty}  N^{-2/3} \|\nabla v_N\|_{L^2(\Omega)}^2\\
			 &= \limsup_{N \to \infty}  N^{-2/3} \left( \|\nabla v_{N,1}\|^2_{L^2(\R^3)} + \| \nabla v_{\infty,3}\|_{L^2(\Omega)}^2\right),
\end{aligned}
\end{align}
where we also used that $N^{-1/3} v_{N,1} \to 0$ strongly in $\dot H^{1}(\R^3 \setminus \Omega)$ in order to replace $\Omega$ by $\R^3$.  Combining  \eqref{eq:energy.sum.2} with \eqref{eq:expansion.0} and \eqref{est:contribution.w} yields \eqref{structure}.
Moreover, by definition of $v_{\infty,3}$ there holds:
\begin{align}
         \|\nabla v_{\infty,3}\|_{L^2(\Omega)}^2&  = \lim_{N \to \infty} N^{1/3} \langle v_{\infty,3}, (\sigma_N - n) e_3 \rangle  \notag \\
         & = \lim_{N \to \infty} \left( N^{1/3}\langle v_{\infty,3}, \bar{\rho}_N e_3 \rangle + N^{1/3}\langle v_{\infty,3}, (\sigma_N - \bar{\rho}_N) e_3 \rangle \right),
\end{align}
where we used again that $\langle v_{N,3}, n e_3 \rangle = 0$. We conclude \eqref{v_3,infty} by observing  that  $N^{1/3} (\sigma_N - \rho_N) \wto 0$ weakly in $\dot H^{-1}(\Omega)$ due to \eqref{est:rhoN.sigma_N}--\eqref{est:rhoN.sigma_N.2}.

Finally, \eqref{weak.v_infty.3} is a consequence of the convergence $N^{-1/3}( v_{N,1} + v_{N,2}) \wto 0$ in $\dot H^1(\Omega)$ from Proposition \ref{pro:contributions.v_N} as well as the bound on $w_N$ that follows from Proposition \ref{prop:variational.psi} and \eqref{est:v_N.in.particles}.
\end{proof}

\subsection{Proof of Proposition \ref{pro:contributions.v_N} } \label{sub_3.1}
Item \ref{it:v_N,3} is independent and proven in a first step. Item \ref{it:v_N,2} is a consequence to the properties of $v_{N,1}$ outside $\Omega$ and is proven in a last step after tackling item \ref{it:v_N,1}. Item \ref{it:v_N.in.particles} will follow from estimates that we show along the proof of items \ref{it:v_N,1}--\ref{it:v_N,3}. All the constants $C$ involved in the following computations are harmless constants. They may depend on the involved exponent $p$ and the constants $c$, $C_0$ and $C_1$ appearing in \eqref{eq:Wasserstein}-\eqref{eq:minimal.distance}-\eqref{eq:Omega.bounded.orthogonal}.
 
\medskip
 
\noindent\emph{Step 1: Proof of \ref{it:v_N,3}:}

To obtain \eqref{est:v_N,3} we proceed in two steps showing in passing the other statements in item \ref{it:v_N,3}. Firstly,  
since $\Omega$ is bounded in one direction (orthogonal to $\xi$) it is standard to adapt the classical construction of solutions to \eqref{def:v_N,3} (see for instance \cite[Section IV.1]{Galdi11}) to yield that $v_{N,3} \in H^1_0(\Omega)$ with 
\[
N^{-1/3}\|v_{N,3}\|_{H^1_0(\Omega)} \leq C
\] 
and claimed convergence when $N\to \infty$ thanks to \eqref{ass:Strong.concergence} and the linearity of the Stokes equations.

Then, we introduce a truncation function $\chi$ such that  $\Omega^{\chi} = \supp(\chi) \cap \Omega$ is $C^2$ and satisfies $\Omega' := \{\chi = 1 \} \subset \Omega^{\chi} \subset \Omega.$ {We set
\begin{align}
v_{N,3}^{\chi} = \chi v_{N,3} - \tilde{v}_{\chi},
 && q_{N,3}^{\chi} = \chi q_{N,3}
\end{align}
where $\tilde{v}_{\chi}$ lifts the divergence of $\chi v_{N,3}$  
in $H^1_0(\Omega^{\chi} \setminus \Omega')$. We have then that \mh{$\dv (v_{N,3}^{\chi}) = 0$}
and
\[
\int_{\Omega^{\chi}} \nabla v_{N,3}^{\chi} : \nabla w 
= \langle N^{\frac 2 3}({\sigma}_N - n) e_3 , \chi w \rangle  +  \int_{\Omega^{\chi}} v \cdot (2 \nabla w \nabla \chi + \Delta \chi w)
- \int_{{\Omega}^{\chi}} \nabla \tilde{v}_{\chi} : \nabla w + \int_{{\Omega}^{\chi}} p \nabla \chi \cdot w.
\]
for any divergence-free $w \in \mathcal C^{\infty}_c(\Omega^{{\chi}}).$ Firstly, we use the embedding $H^{1}_0(\Omega) \subset L^6(\Omega^{\chi})$ to yield that
up to a trivial extension $\tilde{v}_{\chi} \in W^{1,6}_0(\Omega^{\chi})$  (see \cite[Theorem III.3.1]{Galdi11}). Thus, since $p \in L^2(\Omega^\chi)$ (see \cite[Lemma IV.1.1]{Galdi11}) we deduce $v_{N,3}^{\chi} \in W^{1,6}_0(\Omega^{\chi})$ (see \cite[Theorem IV.6.1]{Galdi11}) with bounds that entail
\[
N^{-1/3} \|v_{N,3}\|_{W^{1,6}(\{\chi=1\})} \leq C^{\chi} 
\] 
with $C^{\chi}$ depending furthermore on $\Omega^{\chi}.$
This entails that $v_{N,3} \in L^{\infty}(\{\chi = 1\}) $
We can then reproduce the same argument with a second $\chi$ with support a little smaller to yield:
\[
N^{-1/3} \|\nabla v_{N,3}\|_{L^p(\Omega')} \leq C'
\]
whatever $p \in (1,\infty)$ with the expected dependencies for $C'.$

\medskip		
		
\noindent \emph{Step 2: Proof of \ref{it:v_N,1}:}
The assertion that $N^{-1/3} v_{N,1} \wto 0$ in $\dot H^1(\R^3)$ is a consequence of 
\eqref{eq:convergence.rho_N.sigma_N}. We mention here only that the estimate we derived in $(W^{1,2}(\Omega))^*$ extends straightforwadly into an estimate in the dual of $\dot{H}^1(\mathbb R^3).$ We write then:
	\begin{align}
	v_{N,1} = \sum_{i=1}^N U_i,
\end{align}
where 
 \begin{align}
 - \Delta U_i + \nabla P_i = N^{-1/3}   (\delta_i^{R} - \frac 1 {|Q_i|}\1_{Q_i}) e_3 \quad \dv U_i = 0 \qquad \text{in } \R^3
\end{align}
and $\delta_i^{R} = \mathcal H^2|_{\partial B_i}/{|\partial B_i|}$ is the normalized uniform measure on $\partial B_i$.
Then, 
\begin{align}
	\|\nabla v_{N,1}\|^2_{L^2(\R^3)} = \sum_{i,j=1}^N \int \nabla U_i : \nabla U_j \dd x. 
\end{align}
For the diagonal terms, we split $U_i = U_{i,1} - U_{i,2},$ $P_i = P_{i,1}-P_{i,2}$ corresponding respectively to the solutions of Stokes equations on $\mathbb R^3$ with source terms $N^{-1/3}  \delta_i^R e_3\ $ and   $N^{-1/3}  |Q_i|^{-1} \1_{Q_i} e_3$.  Thanks, to the theory on Stokes problem on $\mathbb R^3$ (see \cite[Section IV.2]{Galdi11}), we know that such solutions can be computed by convolution with a fundamental solution.  We denote $\Phi$ the fundamental solution for the velocity-field. We shall use below extensively that $\Phi$ is $(-1)$-homogeneous (see \cite[Eq. IV.2.3]{Galdi11} for the exact formula). In case of $U_{1,i}$ the existence theory for Stokes problem in exterior domains yields that we have also an exact solution (see \cite[Section V, Eq. (V.0.4)]{Galdi11}). This formula entails in particular that  $U_{i,1} = V_r^{\St}$ in $B_i.$

\medskip

With these remarks at-hand now, we obtain by multiplying the Stokes equations for $U_{i,1}$ with $U_{i,1}$ that:
\begin{align}
	\|\nabla U_{i,1}\|^2_{L^2(\mathbb R^3)} =N^{-1/3} V_{r}^{\St}.   
\end{align}
Moreover, we have, using first the weak formulation of the Stokes equations and then standard estimates for the convolution with $\Phi$ :
\begin{align}
\label{est:L^infty-U_2}
\begin{aligned}
	\|\nabla U_{i,2}\|^2_{L^2} + |(\nabla U_{i,1}, \nabla U_{i,2})_{L^2(\R^3)}| &\leq  N^{-1/3}   \|U_{i,2}\|_{C^0(\overline{Q_i})} \\ &\leq  C N^{-1/3}  N^{2/3} \|\1_{Q_i}\|^{1/3}_\infty  \|\1_{Q_i}\|_1^{2/3} \leq C  N^{-1/3} .
\end{aligned}
\end{align}
Thus, expanding the sum for $U_i$ when computing the $L^2$-norm, we obtain:
\begin{align} \label{est:diagonal}
	\left| \|\nabla U_{i}\|^2_{L^2} - N^{-1/3}  V_{r}^{\St}   \right|  \leq  C  N^{-1/3} .	
\end{align}
Finally, since $U_{i,1}$ is constant in $B_i$, we may reproduce the convolution arguments with $\nabla U_{i,2}$ to yield:
\begin{align} \label{est:diagona.in.particle}
	\|\nabla U_i \|_{L^\infty(B_i)} \leq N^{2/3} \|\1_{Q_i}\|^{2/3}_\infty  \|\1_{Q_i}\|_1^{1/3} \leq C N^{1/3}.
\end{align} 

\medskip

We are now in position to estimate the off-diagonal terms. 
For fixed $i \in \{1,\ldots,N\},$ we consider two cases for index $j \neq i.$ 
Firstly, we say that $Q_j$ is a neighbor of $Q_i$ if $i \neq j$ and $\dist(Q_i,Q_j) \leq C_1$ with $C_1$ being the constant from \eqref{eq:Q_i}. We note that for each $i$ there are at most $M$ neighbors $Q_j$ of $Q_i$ where $M \in \N$ depends only on $C_1$.
We observe now from the explicit formula for $U_{j,1}$
\begin{align}
    |U_{j,1}(x)| \leq C \frac{N^{-1/3}}{|x - X_j|}
\end{align}
for all $x \in \R^3 \setminus B_j$. Thus, combining this with the bound of $U_{j,2}$ derived in \eqref{est:L^infty-U_2}, we have
for all $i \neq j$
\begin{align}
    \|U_{j}(x)\|_{L^\infty(Q_i)} \leq C.
\end{align}
This yields after integration by parts:
\begin{align} \label{est:off-diagonal0}
        \left| 
\sum_{i=1}^N \sum_{Q_j \text{ neighb. } Q_i} \int_{\mathbb R^3} \nabla U_i : \nabla U_j 
\right| \leq N^{-1/3} \left| 
\sum_{i=1}^N \sum_{j \, | \, Q_j \text{ neighb. } Q_i} \|U_j\|_{L^\infty(Q_i)} 
\right| \leq C N^{2/3}
\end{align}

When $Q_j$ is not a neighbor of $Q_i$ we may use the following estimate 
for  smooth test-functions which is reminiscent of  \eqref{eq_expansionordre2}:
\begin{align}
	\langle \delta_i^{R} - |Q_i|^{-1} \1_{Q_i}, \varphi \rangle \leq C N^{-2/3} \|\nabla^2 \varphi\|_{L^\infty(Q_i)}.
\end{align}
Applying this twice,  with $\Phi$:
\begin{align} \label{est:fourth.gradient}
\begin{aligned}
	\left| \int \nabla  U_i : \nabla U_j \dd x\right|  & = N^{-1/3} \left| \langle \delta_i^{R} - |Q_i|^{-1} \1_{Q_i},  U_j \cdot e_3 \rangle \right|
	\\
	& \leq N^{-1} \|\nabla^2 U_j \|_{L^\infty(Q_i)} \\
	&=   N^{-4/3}  \|\langle \delta_j^{R} - |Q_j|^{-1}\1_{Q_j}, \nabla^2 \Phi(x - \cdot) e_3 \rangle\|_{L^\infty(Q_i)} \leq \frac{N^{-2}}{|X_i - X_j|^5}
\end{aligned}
\end{align}
since $|X_i- X_j|$ is comparable to $|x-X_j|$ uniformly in $x \in Q_j$ when $Q_i$
and $Q_j$ are not neighbors.
Thus, 
\begin{align} \label{est:off-diagonal}
	\left| \sum_{i=1}^N \sum_{\substack{i\neq j \\ Q_i \text{ not neighb.} Q_j}} \int_{\mathbb R^3} \nabla  U_i : \nabla U_j \dd x \right| \leq  C \sum_{i=1}^N \sum_{i\neq j}\frac{N^{-2}}{|X_i - X_j|^5} \leq C N^{2/3}.
 \end{align}
Combining \eqref{est:diagonal} and \eqref{est:off-diagonal0}--\eqref{est:off-diagonal} yields \eqref{est:v_N,1}. 

\medskip

Moreover, $X_j$ is in the center of $Q_j$  so that $B_i$ is far from  the support of the convolution defining $U_j$ (the distance scales like $1/N^{1/3}$ with a constant depending on the parameters $c,C_1$ involved in \eqref{eq:minimal.distance}-\eqref{eq:Q_i}).  Arguing as for \eqref{est:fourth.gradient}, we find then a constant $C$ such that,  for $i \neq j$ 
\begin{align}  \label{eq_linfbound}
\|\nabla U_j \|_{L^\infty(B_i)} \leq C \frac{N^{-1}}{|X_i - X_j|^4}   \qquad 
\| \sum_{j \neq i} \nabla U_{j}\|_{L^{\infty}(B_i)} \leq C N^{1/3}.
\end{align}
Combining with \eqref{est:diagona.in.particle} yields
\begin{align} \label{v_N,1.in.particles}
	\|D(v_{N,1}) \|_{L^2(\cup_i B_i)} \leq C r^{3/2} \sup_i \|D(v_{N,1}) \|_{L^\infty(B_i)} \leq C r^{3/2} N^{1/3}.
\end{align}

\medskip

It remains to analyse the convergence of $N^{-1/3}v_{N,1} \to 0 $ outside $\Omega$. % We focus on the $\nabla v_{N,1}$ 
%since similar arguments yield more simply bounds on $v_{N,1}.$
For this, we first provide an $L^{\infty}$-bound that we formulate in the following lemma for future reference:
\begin{lem} \label{prop_linfbound}
Assume that  for all $i=1,\ldots,N$ we have $Q_i \subset  \Omega'$ for some compact $\Omega' \subset \mathbb R^3.$ Then, for any $x \in \mathbb R^3 \setminus \overline \Omega',$ there holds:
\[
\begin{aligned}
\left| \sum_{i=1}^N U_i \right|  & \leq C\left( \mathds{1}_{\dist(x,\Omega') < 2} + \dfrac{\mathds{1}_{\dist(x,\Omega') >1}}{\dist(x,\Omega')^3}\right) \\
\left| \sum_{i=1}^N \nabla U_i \right|  & \leq C\left(  \min  (N^{1/3},   \dist(x,\Omega')^{-1})\mathds{1}_{\dist(x,\Omega') < 2} + \dfrac{\mathds{1}_{\dist(x,\Omega') >1}}{\dist(x,\Omega')^4}\right).
\end{aligned}
\] 
\end{lem}
\begin{proof}
We provide a computation of the second bound since the first one is obtained similarly.
Fix $x \in \mathbb R^3 \setminus \Omega'.$ We have then:
\begin{align}
 \sum_{i=1}^N \nabla U_i 
= \sum_{Q_i \text{neighb.} x} \left( \nabla U_{i,1} + \nabla U_{i,2} \right)
+ \sum_{Q_i \text{ not neighb.} x } \nabla U_i. \label{sum.neighbors}
\end{align}
where we define \enquote{$Q_i$ neighboring $x$} as $\dist(Q_i,x) < 2 (C_1/N)^{1/3}$ (with $C_1$ given in \eqref{eq:Q_i}). 
For the second sum, we proceed similarly to \eqref{eq_linfbound} to obtain that:
\[
\begin{aligned}
\left|\sum_{i | Q_i \text{ not neighb.  of } x } \nabla U_i  \right| 
& \leq \sum_{i | Q_i \text{ not neighb.  of } x } \dfrac{CN^{-1}}{|x-X_i|^4} \leq C \dist(x,  \cup \{ Q_i \text{ not neighb.  of } x \})]^{-1} \\
& \leq C\min  (N^{1/3},   \dist(x,\Omega')^{-1})
\end{aligned}
\]
We note then that we may only have a finite number of indices in the first sum in \eqref{sum.neighbors} and that, for each $i$ neighbor of $x$ there holds:
\[
|\nabla U_{i,1}(x)| \leq \dfrac{CN^{-1/3}}{|x-X_i|^2} \leq CN^{1/3} .   
\]
since $B_i$ is $C/N^{1/3}$ far from $\partial  Q_i.$  We treat the second term with convolution arguments as in \eqref{est:diagona.in.particle} and we obtain $|\nabla U_{i,2}(x)| \leq C N^{1/3}.$
Eventually, we conclude that:
\[
\left| \sum_{i=1}^N \nabla U_i(x)\right| \leq C\left(  \min  (N^{1/3},  \dist(x,\Omega')^{-1}) + N^{1/3}\mathds{1}_{\{\dist(x,\Omega') < (2C_1/N)^{1/3}\}} \right).
\]
We obtain the first bound when  $\dist(x,\Omega') \leq 2.$ When
$\dist(x,\Omega') > 1$ we remark that there are no neighboring $Q_i$ to $x$ and the above computations yield:
\[
\left| \sum_{i=1}^N \nabla U_i(x)\right| \leq \dfrac{C}{N} \sum_{i=1}^{N} \dfrac{1}{|x-X_i|^4}
\]
we conclude by noting that $|x-X_i| \geq \dist(x,\Omega')$ for each $i$ in the sum.
\end{proof}

We continue with the proof of Proposition \ref{pro:contributions.v_N} \ref{it:v_N,1}.
 {Let $\Omega' \subset \Omega$ be chosen independent of $N$ containing all the cubes $Q_i$ which is possible due to assumption \eqref{eq:K}. Then, since $v_{N,1} = \sum_i U_i$, the above lemma implies with dominated convergence that  for arbitrary $p \in (1,\infty)$:}
\begin{align} \label{v_N,1.outside.K}
\lim_{N\to \infty} N^{-1/3} \|\nabla v_{N,1}\|_{L^p(\mathbb R^3 \setminus \Omega')} = 0.
\end{align}
 In particular, we have the same convergence in $W^{1,p}(\mathbb R^3 \setminus \overline{\Omega}).$ 

\medskip

\noindent \emph{Proof of \ref{it:v_N,2}:}
  Using that $v_{N,2}$ is solution to the (homogeneous) Stokes solution inside $\Omega$ (with boundary condition  $-v_{N,1}$ on $\partial \Omega$), we have the variational characterization
\[
\| \nabla v_{N,2}\|_{L^2(\Omega)} = \min \{ \| \nabla v\|_{L^2(\Omega)}:   v \in \dot{H}^{1}(\Omega), ~ \textrm{div} v = 0, ~ v_{|_{\partial \Omega}} = -v_{N,1}  \}.
\]
{To construct a suitable competitor, we consider again a bounded (and connected) set $\Omega'$ as above and set  $v= v_{N,1}$ in $\R^3 \setminus \Omega'$. Inside of $\Omega'$ we then take a divergencefree extension of $v$. It is classical that such an extension can be constructed (e.g. by use of a Bogovk\v{i}i operator) since the condition $\int{\partial \Omega'} v \cdot n = 0$ is satisfied because $\dv v_{N,1} = 0$, and that the extension satisfies
\begin{align*}
    \|\nabla v\|_{L^2(\R^3)} \lesssim C \|\nabla v\|_{L^2(\R^3 \setminus \Omega')} = C \|\nabla v_{N,1}\|_{L^2(\R^3 \setminus \Omega')}
\end{align*}
In view of  \eqref{v_N,1.outside.K}, this concludes the proof of \ref{it:v_N,2}.}
%Thus, we could immediately conclude by item \ref{it:v_N,1} that $N^{-1/3}\nabla v_{N,2} \to 0$ in $L^2({\Omega})$ if there was a continuous extension operator from $H^1(\R^3 \setminus  \Omega)$ to $H^1(\R^3)$. 
% to observe that a continuous extension into the bounded domain $\Omega'$ is sufficient. 

\medskip

\noindent \emph{Proof of \ref{it:v_N.in.particles}:}
 The statement is an immediate consequence of \eqref{v_N,1.in.particles}, item \ref{it:v_N,2} and item \ref{it:v_N,3} applied with $\Omega'$ that contains $K$ from assumption \eqref{eq:K} and with $p$ sufficiently large.

%%%%%%%%%%%%%%%%%%%%%%%%%%%%%

%%%%%%%    Periodic Case %%%%%%%

%%%%%%%%%%%%%%%%%%%%%%%%%%%%%

\section{Explicit computation of the first order correction for periodic configurations} \label{sec:periodic}

In this section, we complete the proof of Theorem \ref{th:expansion} by justifying item \ref{it:thm.periodic}. We will thus assume \eqref{ass:Periodic} throughout this section.
We will assume without loss of generality that $t_d = 0$ in \eqref{ass:Periodic}. Indeed, since we consider the norm of $v_{N,1}$ in the whole space $\R^3$, the shift $t_d$ does not have any influence.

\medskip

We first note that, by classical arguments, there is a unique $v_{per} \in \dot{H}^1(\mathbb T^3_d)$ (homogeneous means here that we consider mean-free functions)
to which we can associate a pressure $p_{per} \in L^2(\mathbb T^3_d)$ such that \eqref{eq_Stokesperiodique} holds true. 
We consider then in analogy to the cubes $Q_i$, $1 \leq i \leq N$ the covering of $\mathbb R^3$ by cubes $(Q_\alpha)_{\alpha\in \mathbb Z^3}$ where $Q_\alpha = N^{-1/3}d (\alpha+ (-1/2,1/2)^3$. Similarly, we adapt the notations
introduced in Section \ref{sub_3.1}: for $\alpha \in \mathbb Z^3,$ $(U_\alpha,P_\alpha)$ is the solution to 
 \begin{align}
 - \Delta U_\alpha + \nabla P_\alpha = N^{-1/3}   (\delta_\alpha^{R} - \frac 1 {|Q_\alpha|}\1_{Q_\alpha}) e_3 \quad \dv U_\alpha= 0 \quad \text{in } \R^3, \qquad  \lim_{|x|\to \infty} |U_\alpha(x)|=0.
\end{align}
and $\delta_\alpha^{R}$ is the normalized uniform measure on $\partial B_\alpha = \partial B_{r N^{-1/3}}(dN^{-1/3}\alpha)$. We keep for technical convenience the labels $\alpha \in \mathbb Z^3$
 We then note by a scaling argument  that:
\[
v_{per}(x) = \sum_{\alpha \in \mathbb Z^3} U_\alpha(N^{-1/3}x) - \fint_{[0,d]^3} U_\alpha(N^{-1/3}y) \dd y  \qquad  v_{N,1} = \sum_{i \in I_{N}} U_{i} + \sum_{i \not\in I_{N}} U_{i},
\]
where we recall the set $I_N$ from \eqref{ass:Periodic} and use the convention that the sums over the index $i$ runs over the set $\{1,\dots,N\}$.

The fact that the first sum converges in $\dot{H}^1(\mathbb T^3_d)$ follows from the decay of $\nabla U_\alpha$ (cf. \eqref{eq_linfbound}).
Let $Z_N \subset \Z^3$ be such that $\cup_{\alpha \in Z_N} Q_\alpha = \cup_{i \in I_N} Q_i = E_N$
with $E_N$ as in \eqref{ass:Periodic}.  We then obtain then:
\[
\|\nabla v_{N,1}\|^2_{L^2(\mathbb R^3)} = N^{2/3} \|\nabla v_{per}(N^{1/3}x) \|^2_{L^2(E_N)} +  rem_{1,N} + rem_{2,N} 
\]
where
\begin{align*}
rem_{1,N} &= -\| \sum_{\alpha \notin Z_{N}} \nabla  U_\alpha \|^2_{L^2(E_N)} - 2 \int_{E_N} N^{1/3}\nabla v_{per}(N^{1/3}x) :  \sum_{\alpha \notin Z_N}   \nabla U_\alpha \dd x \\
rem_{2,N} & = \| \sum_{i \in I_{N}}  \nabla U_i\|^2_{L^2(\mathbb R^3 \setminus E_N)} + \|\sum_{i \notin I_{N}} \nabla  U_i \|^2_{L^2(\mathbb R^3)} + 2 \int_{\R^3} \nabla v_{N,1} : \sum_{i \notin I_{N}}  \nabla U_i \dd x
\end{align*}
By standard arguments, we have:
\[
\lim_{N\to \infty} \|\nabla v_{per}(N^{1/3}x) \|^2_{L^2(E_N)}  = \|v_{per}\|^2_{\dot{H}^1(\mathbb T^3_d)}
\]
and the the second identity in \eqref{char.periodic} yields from the analysis of the periodic problem in \cite{Hasimoto59}.
Our proof thus reduces to obtaining that:
\[
\limsup_{N\to \infty} N^{-2/3} (rem_{1,N} + rem_{2,N}) =0.
\]

Concerning $rem_{1,N}$ we note that we have first the bound:
\[
|rem_{1,N}| \leq C  \|\nabla \sum_{\alpha \notin Z_N}  U_{\alpha} \|_{L^2(E_N)}  \left(  1 +  \|\nabla \sum_{\alpha \notin Z_N}  U_{\alpha} \|_{L^2(E_N)} \right) 
\]
Then, we apply Lemma \ref{prop_linfbound} to yield that, for arbitrary $x \in E_N$ there holds:
\[
\sum_{\alpha \notin Z_N}  |\nabla   U_\alpha(x)| \leq C \min  (N^{1/3},  \dist(x,\partial E_N)^{-1}).
\]
We note here that, to apply properly Lemma \ref{prop_linfbound}  we must invoke  an \enquote{invading domain} argument and firstly approximate the infinite sum by 
finite sums. The above bound yields from the remark that the right-hand side does not depend on the finite subset of $\mathbb Z^3 \setminus Z_N$ that we would choose. Recalling $E_N \subset E_{N+1}$ from \eqref{ass:Periodic} and that on the other side the sets $E_N$ are contained in a compact set independently of $N$ due to \eqref{eq:K}, we deduce with the dominated convergence theorem
\[
\limsup_{N\to \infty} N^{-1/3}\|\nabla \sum_{\alpha \notin Z_{N}}   U_{\alpha} \|_{L^2(E_N)}  = 0,  \qquad \limsup_{N \to \infty}N^{-2/3} rem_{1,N}  =0.
\]

Concerning $rem_{2,N}$, we can get similarly as above
\[
\limsup_{N\to \infty} N^{-1/3} \| \sum_{i \in I_{N}}  \nabla U_i\|^2_{L^2(\mathbb R^3 \setminus E_N)} = 0.
\]
Moreover, since by assumption \eqref{ass:Periodic} $ \#\{ i \notin I_N \} \ll N$, we have from the bound on $U_i$ in \eqref{est:diagona.in.particle}
that
\begin{align}
    N^{-1/3} \|\sum_{i \notin I_{N}} \nabla  U_i \|_{L^2(\mathbb R^3)} = 0.
\end{align}
Combining these estimates yields $\lim_{N \to \infty} N^{-2/3} rem_{2,N} = 0$
which concludes the proof.

%%%%%%%%%%%%%%%%%%%%%%%%%%%%%

%%%%%%%   Ill-Prepared Case %%%%%%%

%%%%%%%%%%%%%%%%%%%%%%%%%%%%%

\section{Computations in the ill-prepared case} \label{sec:unfavorable}

We provide here the computations in the ill-prepared case when \eqref{eq:gradient} is not satisfied.
Let $(u_N,p_N)$ be the solution to \eqref{def:u_N}. We introduce  again $(v_N,q_N)$ the solution to 
\begin{align} \label{app_def:v_N}
\left.
\begin{aligned}
 	- \Delta v_N + \nabla q_N &=  N^{2/3} \bar \rho_N e_3 &&  \text{in } \Omega\\
  \dv v_N &= 0&& \text{in } \Omega \\
      v_N &= 0&& \text{on } \partial \Omega
      \end{aligned}
      \right\}
 \end{align}
and $w_N = u_N - v_N.$ We point out that, without assumption \eqref{eq:gradient} we may not normalize the pressure to add the $-ne_3$ term to the right-hand side without modifying $v_N$.   
% Below,  we split $v_N$ further into $v_N = v_{N,1} + v_{N,2}$ where $v_{N,1}$ solves the Stokes system on $\mathbb R^3$
% \begin{align} \label{app_def:v_N,1}
% \left.
% \begin{aligned}
%  	- \Delta v_{N,1} + \nabla q_{N,1} &=  N^{\frac 2 3}\bar \rho_N  e_3 && \text{in } \R^3\\
%   \dv v_{N,1} &= 0 &&  \text{in } \R^3
% \end{aligned}
% \right\}
%  \end{align}
%  and $(v_{N,2},q_{N,2})$ complies with the boundary condition:
%  \begin{align} \label{app_def:v_N,2}
%  \left.
%  \begin{aligned}
%  	- \Delta v_{N,2} + \nabla q_{N,2} &= 0  &&  \text{in } \R^3 \\
%   \dv v_{N,2} &= 0 && \text{in } \R^3 \\
%   v_{N,2} &= -v_{N,1} && \text{on } \partial \Omega,
%   \end{aligned}
%   \right\}
%  \end{align}
%  \rhcomment{It seems that $v_{N,i}$ are never used in this section}
 
The main goal of this section is a proof of item \ref{it:thm.unfavorable} in Theorem \ref{th:scaling.N}.  We complement the proof with a more refined description of $\bar{V}_N$ at the end of this section.  To achieve our main goal we first provide the following proposition:
 
 \begin{prop} \label{prop:limit.unfavorable}
The vector-fields  $v_N$ and $w_N$ introduced above satisfy the following statements:
 \begin{itemize}[(i)]
\item there exists $(v_\ast,q_\ast) \in H^1_0(\Omega) \times H^{-1}(\Omega)$
for which $ N^{-2/3} v_N \to v_\ast$ in $H^1_0(\Omega)$ and:
\begin{align} \label{def:v_ast}
\left.
\begin{aligned}
 	- \Delta v_\ast + \nabla q_\ast &= ne_3 &&  \text{in } \Omega, \\
  \dv v_\ast &= 0 &&  \text{in } \Omega.
  \end{aligned}
  \right\}
 \end{align}
 \item there exists a constant $C$ which depends only on $c$ from \eqref{eq:minimal.distance} and $\Omega$ such that, for $N$ sufficiently large:
 \begin{align} \label{est:w_N.unfavorable}
 	\|w_N\|_{H^1(\Omega)} \leq C N^{2/3} r^{3/2}.
 \end{align}
 \end{itemize}
\end{prop}
\begin{proof}
Item $i)$ is a direct consequence to \eqref{est:n.bar.rho_N} in {Proposition \ref{pro:Poincare}} by standard arguments on generalized solutions to Stokes system (see \cite[Theorem IV.1.1]{Galdi11}).  We point out that the result holds actually whether data are well-prepared or ill-prepared.  The main difference between the ill-prepared and well-prepared setting is that $v_*=0$ in the latter one. 

\medskip

Since $v_N/N^{2/3}$ converges to $v_*$ in $H^{1}_0(\Omega),$ we can bound $w_N$ as follows thanks to Proposition~\ref{prop:variational.psi}:
\[ \|\nabla w_N\|_{L^2(\Omega)} \leq C
\|D(v_N)\|_{L^2(\cup B_i)} \leq C N^{2/3} \left( \|D(v_*)\|_{L^2(\cup B_i)} + C\|N^{-2/3} v_N - v_{*}\|_{H^1(\Omega)} \right) 
\]
where the second term in the parenthesis can be made arbitrary small for $N$ large.  We remark then that $n \in L^{\infty}(\Omega)$ so that standard elliptic regularity results entail in particular that  $v_* \in W^{2,4}(\Omega')$ for arbitrary bounded $\Omega' \subset \Omega$ and thus $v_*  \in C^1(\bar{\Omega}).$ We infer then that the first term in the parenthesis is bounded by $r^{3/2}. $ This ends the proof.
\end{proof}

\begin{proof}[Proof of Theorem \ref{th:scaling.N}\ref{it:thm.unfavorable}] 
Assume \eqref{eq:gradient} is not satisfied. Then, by combining Proposition \ref{prop:limit.unfavorable} and \eqref{est:n.bar.rho_N} in {Proposition \ref{pro:Poincare}} (which implies the strong convergence of $\bar{\rho}_N$ to $n$ in $H^{-1}(\Omega)$),   we infer:
\begin{align*}
	N^{-2/3}  \bar V_N  & = N^{-2/3}  \langle v_N, \bar \rho_N \rangle_{H^1,H^{-1}} +  N^{-2/3} \langle w_N, \bar \rho_N \rangle_{H^1,H^{-1}}\\
	& \to   \langle  v_\ast, n \rangle_{H^1,H^{-1}} + O(r^{3/2}) =  \int_\Omega v_\ast  n + O(r^{3/2}),
\end{align*}
which yields \eqref{est:V.unfavorable} since $\|v_\ast\|_{H^1(\Omega)} \leq C \|n\|_{L^{\infty}(\Omega)}.$
Moreover,  we have via a standard energy estimate:
\begin{align}
	\int_\Omega v_\ast  \cdot n  e_3 = 
	  \| \nabla v_\ast\|_{L^2(\Omega)}^2,
\end{align}
which yields \eqref{est:V^sed.unfavorable} since $v_\ast \neq 0$ if \eqref{eq:gradient} is not satisfied. This ends our proof.
\end{proof}

To complement the analysis of the ill-prepared case, we provide a sharper description of  $\bar{V}_N$ for large values of $N.$  For this, we introduce  further notations for solutions to \eqref{def:v_ast}.  Indeed, we remark that this solution is fixed by the vector $e_3$ so that changing this value to another vector $\tilde{e} \in \mathbb R^3$ would yield a different velocity-field.  Below, we highlight this possible dependency by writing $v_*[\tilde{e}]$ the solution associated with the vector $\tilde{e} \in \mathbb R^3.$   We can now state our main proposition:
 \begin{prop}
Assume that \eqref{eq:gradient} does not hold.  Then, there exists $V_* \in \mathbb R^3$ such that:
 \begin{itemize}[(i)]
 \item there exists a constant $C$ independent of $N \in \mathbb N$ sufficiently large  for which 
 \[
 \bar{V}_N = N^{\frac 23} \left( V_* + rem_{N}\right) \quad \text{ with } \quad |rem_N| \leq C r.
 \]
 
 \item there holds:
 \[
 V_* \cdot  e= \int_{\Omega} \nabla v_*[e_3] : \nabla v_*[e] \quad \forall  \, e \in \mathbb S^{2}  
 \]
 \end{itemize}
\end{prop} 

\begin{rem}
We first point out that $v_*$ does not depend on $r.$ So, when $\nabla n \times e \neq 0$ we have indeed captured the first order of $\bar{V}_N $ with a remainder smaller than $O(r).$ We note that we can use the system satisfied by $v_*$ to rewrite:
\[
V_* \cdot e = \int_{\Omega} v_*[e_3] \cdot ne = \int_{\Omega} v_*[e] \cdot n e_3 
\]
We also recall that, in the degenerate case $\nabla n \times e = 0,$ there holds $v_*[e] = 0.$ In this case,  the computations of the previous section hold and show that  $\bar{V}_N \cdot e \leq C N^{1/3}$ similarly as we obtained \eqref{est:V.favorable}. In particular, the results obtained in the present section are  not optimal in this direction. 
\end{rem}

\begin{proof}
We prove that, for arbitrary $e \in \mathbb S^2,$ there holds:
\[
\bar{V}_N \cdot e = N^{\frac 23} \left(  \int_{\Omega} \nabla v_*[e_3] : \nabla v_*[e] + O(r) \right )
\]
This shall complete the two items of the proposition simultaneously. 

\medskip

Given $e \in \mathbb S^2,$ let us denote by $u_N[e],v_N[e],w_N[e]$ the velocity-fields associated to the problem  \eqref{def:u_N} replacing $e_3$ by $e$
(analogously as the notation $v_*[e]$ introduced above).  By Proposition \ref{prop:limit.unfavorable},  we have: 
\[
\| \nabla v_N[e_3]\|_{L^2} + \|\nabla v_N[e]\|_{L^2} \leq C N^{2/3}, \qquad 
\| \nabla w_N[e_3]\|_{L^2} \leq C N^{2/3} r^{\frac 32}
\]
and 
\[
N^{-\frac 23}v_N[e_3] \to v_*[e_3]
\qquad 
N^{-\frac 23}v_N[e]  \to v_*[e] 
\quad \text{ in $H^1_0(\Omega).$}
\]
Furthermore,  like in the previous proof,  there holds:
\begin{align*}
N^{-\frac 23} \bar{V}_N \cdot e &
= N^{-\frac 43} \langle v_N[e_3],  - \Delta v_N[e] + \nabla q_N[e]\rangle_{H^1_0(\mathbb R^3),H^{-1}(\mathbb R^3)} +  N^{-\frac 23}\langle w_N[e_3], \bar \rho_N  e \rangle_{H^1_0(\Omega),H^{-1}(\Omega)} \\
& = {N^{-\frac 43}} \int_{\Omega} \nabla v_N[e_3] :  \nabla v_N[e]
+ N^{-\frac 23} \langle w_N[e_3], \bar \rho_N e \rangle_{H^1_0(\Omega),H^{-1}(\Omega)}
\end{align*}
We conclude by remarking that by Proposition \ref{prop:limit.unfavorable}
\[
\lim_{N\to \infty} \int_{\Omega} \nabla v_N[e_3] :  \nabla v_N[e] =  \int_{\Omega} \nabla v_*[e_3] :  \nabla v_*[e]
\]
and, by \eqref{est:rhoN.sigma_N}  and \eqref{est:n.sigma_N} in {Proposition \ref{pro:Poincare}} and the above bound on $w_N,$ that:
\[
{N^{-\frac 23}} \left|  \langle w_N[e_3], \bar \rho_N  e \rangle_{H^1_0(\Omega),H^{-1}(\Omega)} \right| \leq  N^{-\frac 23} \|\bar \rho_N\|_{H^{-1}(\Omega)} \|w_N[e_3]\|_{H^{1}(\Omega)} \leq  Cr.
\]
%\rhcomment{I don't see where the $\eps$ comes from. I guess we want to combine \eqref{est:rhoN.sigma_N} and \eqref{est:n.sigma_N} and the fact that $\|n\|_{H^{-1}} \lesssim 1$ to see that  $\limsup_{N \to \infty} \|\bar \rho_N\|_{H^{-1}(\Omega)} \lesssim 1$? }
%for any $\varepsilon>0.$ This ends the proof.
\end{proof}

\section*{Acknowledgements}

The two authors warmly thank David Gérard-Varet and Amina Mecherbert for fruitful discussions during the preparation of this paper.  

R.H. thanks Juan Velázquez for discussions that have played an important part in discovering the nature and significance of this problem.

R.H. has been supported  by the German National Academy of Science Leopoldina, grant LPDS 2020-10.
Moreover, R.H. acknowledges support by  the Deutsche Forschungsgemeinschaft (DFG, German Research Foundation) 
through the collaborative research center ``The Mathematics of Emerging Effects'' (CRC 1060, Projekt-ID 211504053) 
and the Hausdorff Center for Mathematics (GZ 2047/1, Projekt-ID 390685813). 

M.H. acknowledges support of the Institut Universitaire de France and project \enquote{SingFlows} ANR-grant number: ANR-18-CE40-0027.   This paper was partly written while M.H. was benefiting a \enquote{subside à savant} from Université Libre de Bruxelles.  He would like to thank  the mathematics department at ULB for its hospitality.

\printbibliography

\end{document}